\documentclass[10pt, a4paper]{amsart}
\usepackage[utf8]{inputenc}
\usepackage[T1]{fontenc}

\usepackage{paralist}

\usepackage{amsmath}
\usepackage{amssymb}
\usepackage{amsthm}
\usepackage{graphicx}
\usepackage{url}
\usepackage{enumerate}
\usepackage{xcolor}
\usepackage{mathrsfs} 
\usepackage{mathtools}
\usepackage{t1enc}
\usepackage{enumitem}

\newtheorem{thm}{Theorem}
\newtheorem{lem}[thm]{Lemma}
\newtheorem{prop}[thm]{Proposition}
\newtheorem{cor}[thm]{Corollary}

\theoremstyle{definition}

\newtheorem{rem}[thm]{Remark}

\DeclareMathOperator{\Pe}{\mathcal{P}} 
\DeclareMathOperator{\Emp}{\mathfrak{m}} 

\newcommand{\dbar}{\bar{d}}
\DeclareMathOperator{\dist}{dist}

\newcommand{\di}{D_P}

\DeclareMathOperator{\sint}{\overline{\kern-5pt\int}} 
\DeclareMathOperator{\m}{\mathfrak m}

\newcommand{\und}[1]{\underline{#1}}

\newcommand{\om}[1]{\hat{\omega}_{\mathcal F}({#1})}

\newcommand{\Folner}{F{\o}lner }
\renewcommand{\phi}{\varphi}

\DeclareMathOperator{\M}{\mathcal{M}}

\DeclareMathOperator{\F}{\mathcal{F}}
\newcommand{\HD}{H}
\DeclareMathOperator{\Bl}{\mathcal{B}}
\DeclareMathOperator{\lang}{\Bl}

\newcommand{\alf}{\mathscr{A}}

\newcommand{\htop}{h_\textrm{top}}

\DeclareMathOperator{\Fin}{Fin}

\newcommand{\DW}{D_W}

\def\ocirc#1{\ifmmode\setbox0=\hbox{$#1$}\dimen0=\ht0
    \advance\dimen0 by1pt\rlap{\hbox to\wd0{\hss\raise\dimen0
    \hbox{\hskip.2em$\scriptscriptstyle\circ$}\hss}}#1\else
    {\accent"17 #1}\fi}

\newcommand{\eps}{\varepsilon}
\newcommand{\R}{\mathbb{R}}

\newcommand{\N}{\mathbb{N}}

\author{Martha {\L}{\k{a}}cka \and Marta Straszak}

\address[M. {\L}{\c{a}}cka]{
Faculty of Mathematics and Computer Science, Jagiellonian University in Krakow, ul. \L o\-jasiewicza 6, 30-348 Krak\'ow, Poland}\email{martha.ubik@gmail.com}
\urladdr{www2.im.uj.edu.pl/MarthaLacka/}
\address[M. Straszak]{
Faculty of Mathematics and Computer Science, Jagiellonian University in Krakow, ul. \L o\-jasiewicza 6, 30-348 Krak\'ow, Poland}\email{marta.pietrzyk@doctoral.uj.edu.pl}

\title[Quasi-uniform Convergence...]{Quasi-uniform Convergence in Dynamical Systems\\Generated by an Amenable Group Action}
\date{\today}
\begin{document}
\maketitle
\begin{abstract}
We study the Weyl pseudometric associated with an action of a countable amenable group on a compact metric space. We prove that the topological entropy and the number of minimal subsets of the closure of an orbit are both lower semicontinuous with respect to the Weyl pseudometric.
Furthermore, the simplex of invariant measures supported on the orbit closure varies continuously.
We apply the Weyl pseudometric to Toeplitz configurations for arbitrary amenable residually finite groups. We introduce the notion of a regular Toeplitz configuration and demonstrate that all regular Toeplitz configurations define minimal and uniquely ergodic systems. We prove that this family is path-connected in the Weyl pseudometric. This leads to a new proof of a theorem of F.~Krieger, saying that Toeplitz configurations can have arbitrary finite entropy.
\end{abstract}

\section{Introduction}
Given a dynamical system one may measure the distance between points by computing an averaged distance along their orbits. This defines a pseudometric and it is expected that it may be useful to describe the dynamics of the system.
An example is the \emph{Weyl pseudometric} 
\[D_W(x,y)=\limsup\limits_{N\to\infty}\frac{1}{N}\sup_{k\in\mathbb Z}\sum_{i=0}^{N-1}\rho(T^{k+i}(x),T^{k+i}(y)),\]
where $x,y\in X$, $(X,\rho)$ is a compact metric space and $T\colon X\to X$ is a homeomorphism.
We call the convergence with respect to $D_W$ the \emph{quasi-uniform convergence}.
This term was introduced by Jacobs and Keane in \cite{JK}, who proved that if $\{x_n\}_{n\in\N}$ tends to $x$ in $D_W+\rho$ and all $x_n$'s are strictly transitive (that is there is only one Borel probability invariant measure on the closure of its orbit), then so is $x$. Further dynamical consequences of the quasi-uniform convergence were studied by Downarowicz and Iwanik \cite{DI}, Blanchard, Formenti and Kůrka \cite{BFK}, and Salo and T{\"o}rm{\"a} \cite{ST}.

The  Weyl pseudometric can be extended to amenable group actions. We study the properties of the quasi-uniform convergence in this wider setting. We prove that the topological entropy and the number of minimal subsystems of the closure of an orbit of a point are both lower semicontinous functions in the Weyl pseudometric.
We show that the simplices of invariant measures supported on orbit closures vary continuously (when we endow the family of all non-empty compact sets of invariant measures with a natural topology). We also prove that for shift actions the entropy of the orbit closure is a $D_W$-continuous function. This extends results from \cite{DI} to amenable group actions.

We apply our results to systems defined as the closure of an orbit of a Toeplitz configuration over $G$, which (for $G=\mathbb Z$) were introduced by Jacobs and Keane in~\cite{JK} and have already proved to have numerous remarkable properties \cite{Williams}. Using Toeplitz configurations one can construct for instance strictly ergodic systems with positive entropy or minimal systems which are not uniquely ergodic \cite{Williams}. 

Toeplitz configurations for amenable residually finite group actions have also been studied extensively \cite{CortezPetite, Cortez, KriegerFR, Krieger}. Here we use the Weyl pseudometric to introduce regular Toeplitz configurations for such groups actions and show that they generate strictly ergodic systems with zero topological entropy.

We also show that the family of Toeplitz configurations is path-connected with respect to the Weyl pseudometric. This enables us to establish paths of Toeplitz configurations with continuously varying entropy. Thanks to that we prove the Krieger theorem \cite[Theorem 1.1]{KriegerFR} stating that if $G$ is a residually finite countable amenable group then for any number $t$ in $[0,\log k)$ there exists a Toeplitz configuration over $k$-letter alphabet with entropy $t$. This in particular means that there are uncountably many pairwise not conjugated Toeplitz dynamical systems.

\section{Notation, Basic Definitions and Setting}
Throughout this paper $G$ denotes a discrete countable amenable\footnote{We omit this assumption at the beginning of Section~\ref{section7}.} group with an identity element $e$, $(X,\rho)$ is a compact metric space with diameter bounded by $1$ and $\Fin(G)$ is the family of all non-empty finite subsets of $G$. If $G$ acts on $X$, that is, there is a continuous map $G\times X\to X$ such that $(e,x)=x$ for every $x\in X$, and $(g,(h,x))=(gh,x)$ for every $g,h\in G$ and $x\in X$, then we call the pair $(X,G)$ a dynamical system. 


We characterize countable amenable groups using \Folner sequences.  Recall that a sequence $\{F_n\}_{n\in\N}\subset\Fin(G)$ is a \emph{(left)\footnote{We used the adjective to distinguish \emph{left} \Folner sequences from their right- and two-sided analogues. In a similar vain notions related to non-commutative groups have ,,left'' and ,,right'' variants.  For brevity, we will use adjectives ,,left/right'' only in definitions and routinely omit them later, since the choice of ,,left/right'' is fixed throughout the paper.
} F{\o}lner sequence} if
\[
\lim_{n\to\infty}|gF_n\triangle F_n|/|F_n|=0\quad\text{for every $g\in G$.}
\]
We say that $G$ is \emph{amenable}, if it admits a (left) F{\o}lner sequence.
Given $F,K\in\Fin(G)$ and $\eps>0$ we say that $F$ is \emph{$(K,\eps)$-invariant} if $|KF\triangle F|\leq\eps|F|$. Note that if $\mathcal F=\{F_n\}_{n\in\N}$ is a F{\o}lner sequence, then for every $\eps>0$ and $K\in\Fin(G)$ there exists $N\in\N$ such that $F_n$ is $(K,\eps)$-invariant for every $n\geq N$.
By $\F$ we always denote a \Folner sequence and if we write $ \sup_{\mathcal{F}}$ or $\cup_{\mathcal{F}}$, we mean that the supremum or the union is taken over all \Folner sequences.


Let $X^G$ be the set of all  functions from $G$ to $X$.
For any $x\in X$ we denote by $Gx$  the \emph{orbit} of $x$, that is $Gx=\{gx:{g\in G}\}$, and
by $\und{x}_G\in X^G$, the \emph{trajectory} of $x$ i.e. $\{\und{x}_G\}_g=gx$.

Let $\M(X)$ be the set of Borel probability measures on $X$ equipped with the weak$\star$ topology given by the compatible \emph{Prokhorov metric} 
\[
\di(\mu,\nu)=\inf\{\eps>0\,:\,\text{for every Borel set }B\subset X\text{ one has }\mu(B)\leq\nu(B^{\eps})+\eps\},
\]
where $B^{\eps}=\{y\in X\,:\,\text{there exists }x\in B\text{ such that }\rho(x,y)<\eps\}$ and $\mu,\nu\in\M(X)$.
Thus $(\M(X), \di)$ is a compact metric space. Let $\M_G(X)\subseteq\M(X)$ be the simplex of $G$-invariant measures on $X$. Amenability of $G$ implies that $\M_G(X)$ is non-empty. A measure $\mu\in\M_G(X)$ is ergodic if for every $G$-invariant set $A\subset X$ one has either $\mu(A)=0$ or $\mu(A)=1$. We denote by $\M_G^e(X)\subseteq\M_G(X)$ the family of all ergodic measures. A point $x\in X$ is \emph{generic} for $\mu\in\M_G(X)$ along $\mathcal F=\{F_n\}_{n\in\N}$ if 
\[\frac{1}{|F_n|}\sum_{f\in F_n}\hat\delta_{fx}\to\mu\text{ as }n\to\infty,\]
where $\hat\delta_{fx}$ denotes the Dirac measure supported at $fx$.
A $\mu$-generic point $x\in X$ is \emph{regular} if it belongs to the support of $\mu$. 

A \Folner sequence $\mathcal F=\{F_n\}_{n=1}^{\infty}$ is \emph{tempered} if for some $C>0$ and all $n\in\N$ one has
$|\bigcup_{k\leq n}F_k^{-1}F_{n+1}|\leq C\cdot\left|F_{n+1}\right|$.
Every F{\o}lner sequence has a~tempered subsequence \cite[Proposition 1.5.]{Lindenstrauss01}. It follows from the pointwise ergodic theorem for amenable group actions \cite[Theorem 1.3.]{Lindenstrauss01} that given a tempered F{\o}lner sequence $\mathcal F$ and $\mu\in\M_G^e(X)$ one can find a generic point for $\mu$ along $\mathcal F$.

The Hausdorff distance $\HD$ between two non-empty compact subsets $A,B$ of a metric space $(X,\rho)$ is given by
\[\HD(A,B)=\max\big\{\inf\{\eps>0\,:A\subset B^{\eps}\},\inf\{\eps>0\,:B\subset A^{\eps}\}\big\}.\]
We consider $\M(X)$ as a metric space with the Hausdorff distance.

Recall that a pseudometric on a set $Y$ is a function $p\colon Y\times Y\to\mathbb R_+$ which obeys the triangle inequality and is symmetric.  We say that $f\colon (Y,p)\to \mathbb R\cup\{\infty\}$ is \emph{lower semicontinuous} if whenever $p(x_n,x)\to 0$ as $n\to\infty$ one has $\liminf\limits_{n\to\infty}f(x_n)\geq f(x)$. 

We say that a set $S\subset G$ is \emph{(left) syndetic} if there exists $F\in\Fin(G)$ such that $FS=G$. A set $T\subset G$ is \emph{(right) thick} if for every $F\in\Fin(G)$ there exists $\gamma\in T$ such that $F^{-1}\gamma\subset T$. The properties of being syndetic and thick are dual:
\begin{itemize}[leftmargin=*]
\item  $S\subset G$ is left syndetic if and only if $S\cap T\neq\emptyset$ for every right thick  $T\subset G$.
\item $T\subset G$ is right thick if and only if $T\cap S\neq\emptyset$ for every left syndetic $S\subset G$.
\end{itemize}
In this paper by \emph{syndetic} and \emph{thick} we always mean \emph{left syndetic} and \emph{right thick}.

A countable group $G$ is \emph{residually finite} if there exists a nested sequence of finite index normal subgroups whose intersection is trivial. We write $H<_f G$ ($H\triangleleft_f G$) if $H$ is a finite index (normal) subgroup of $G$.

A point $x\in X$ is called \emph{transitive} if the set $Gx$ is dense in $X$. We say that $(X,G)$ is \emph{minimal} if every point in $X$ is transitive.

Let $\alf$ be a finite discrete space and $\alf^G$ be a product space. By $(\alf^G,G)$ we mean a dynamical system determined by the shift action $(h,(x_g)_{g\in G})\mapsto (x_{gh})_{g\in G}.$
For $a\in\alf$ let $[a]=\{\und x\in\alf^G\,:\,x_e=a\}$. We equip $\alf^G$ with a compatible metric such that $\dist([a],[b])>1/2$ for $a\neq b$.

Let $\Psi$ be a family of functions from a (pseudo)metric space $(Y,p_Y)$ to a (pseudo)\-metric space $(Z,p_Z)$. We say that $\Psi$ is \emph{uniformly equicontinuous} if for every $\eps>0$ there exists a function $\mathbb R_+\ni\eps\to \delta(\eps)\in\mathbb R_+$ such that for every $f\in\Psi$ and all $y_1,y_2\in Y$ satisfying $p_Y(y_1,y_2)<\delta(\eps)$ one has $p_Z(f(y_1), f(y_2))<\eps$. We call the function $\delta$ a \emph{modulus of equicontinuity} of $\Psi$. If $\{\Psi(k)\,:\,k\in\mathcal K\}$ is a family of families of functions, we say that \emph{modulus of equicontinuity do not depend on $k\in\mathcal K$} if there exists a function $\delta$ which is a modulus of equicontinuity of $\Psi(k)$ for every $k\in\mathcal K$. Analogously we define a \emph{modulus of uniform equivalence} of (pseudo)metrics.
\section{Besicovitch Pseudometric and Weyl Pseudometric}
The main aim of this section is Theorem~\ref{tog} in which we express the Weyl pseudometric in several equivalent ways. But first we recall some notions useful in the further considerations.
\subsection{Densities}
For $F\in\Fin(G)$ and $A\subset G$ define
\[
D_F(A)=|A\cap F|/|F|, \quad D_{*}^F(A)=\inf_{g\in G} D_{Fg}(A)\quad\text{and}\quad D^*_F(A)=\sup_{g\in G} D_{Fg}(A).
\]
The \emph{upper asymptotic density} of $A$ with respect to $\mathcal F=\{F_n\}_{n\in\N}$ is given by
\[\overline{d}_{\mathcal F}(A)=\limsup_{N\to\infty}D_{F_N}(A).\]
The \emph{upper Banach density} of $A$ is defined as $D^*(A)=\inf\{D^*_F(A)\,:\,F\in\Fin(G)\}$.
If $\mathcal F=\{F_n\}_{n\in\N}$ is a \Folner sequence, then we have (see \cite[Lemma 2.9]{DHZ})
\[
D^*(A)=\lim_{n\to\infty} D^*_{F_n}(A).
\]
In particular, the above limit exists and does not depend on the choice of $\mathcal F$. We also have \cite[Lemma 3.3]{BBF} that
\[
D^*(A) = \sup_{\F} \limsup_{n\to\infty}D_{F_n}(A).
\]
The \emph{lower Banach density} of $A$ is given by
$D_*(A)=1-D^*(G\setminus A)$. Note that \begin{multline*}
D_*(A)=\sup\big\{D_*^F(A)\,:\,F\in\Fin(G)\big\}=\lim\limits_{n\to\infty}D^{F_n}_*(A)=\\=\inf\left\{\limsup\limits_{n\to\infty}D_{H_n}(A)\,:\,\{H_n\}_{n\in\N}\text{ is a \Folner sequence}\right\},
\end{multline*}
where $\mathcal F=\{F_n\}_{n\in\N}$ is any \Folner sequence.
If $D_*(A)=D^*(A)$, then we say that $A$ has the \emph{Banach density} $D(A)=D^*(A)$.
\subsection{Besicovitch Pseudometric}
This way of measuring distance between trajectories is motivated by a notion used by Besicovitch in his studies of almost periodic functions. The Besicovitch pseudometric appeared also in \cite{Auslander59,Fomin,FGJ,Oxtoby52}.

The Besicovitch pseudometric on $X^{G}$ along $\mathcal F=\{F_n\}_{n\in\N}$, denoted by $D_{B,\mathcal F}$, is defined for $\und{x}=\{x_g\}_{g\in G}$, $\underline{x}'=\{x'_g\}_{g\in G}\in X^{G}$ by
	\[
D_{B,\mathcal F}(\underline x,\underline x')=\limsup_{N\to\infty}\frac{1}{|F_N|}\sum_{g\in F_N}\rho(x_g,x_g').
	\]
For $x,x'\in X$ we define $D_{B,\mathcal F}(x,x')$ to be equal the $D_{B,\mathcal F}$-distance between trajectories of $x$ and $x'$.





Repeating the proof of \cite[Lemma 2]{KLOg} we get the following:
\begin{lem}\label{lem:db-prim}
Fix $\mathcal F$ and define
$D'_B$ on $X^G$ as
\[
D'_{B,\mathcal F}(\und{x}_G,\und{z}_G)=\inf
\{\delta>0: \dbar_{\mathcal F}(\{g\in G: \rho(x_g,z_g)\ge \delta\})<\delta\}.
\]
Then $D_{B,\mathcal F}$ and $D_{B,\mathcal F}'$ are uniformly equivalent on $X^G$. Moreover, the modulus of uniform equivalence does not depend on the choice of a F{\o}lner sequence.
\end{lem}


\begin{cor}
Let $\tilde\rho$ be a compatible metric on $X$ and $\tilde D_{B,\mathcal F}$ be defined as $D_{B,\mathcal F}$ above with $\tilde\rho$  in place of $\rho$. Then
$\tilde D_{B,\mathcal F}$ and $D_{B,\mathcal F}$ are uniformly equivalent on $X^G$. 
\end{cor}

\subsection{Weyl Pseudometric}
After some preparations we characterize Weyl pseudometric and its relatives in Theorem~\ref{tog}.

For $\und{x},\und{x}' \in X^G$ and $F \in \Fin(G)$ define
\[
\Delta_F(\und{x},\und{x}') = \sum_{f\in F} \rho (\und{x}_f,\und{x}_f')\quad\text{and}\quad\Delta_F^*(x,x') = \sup_{g \in G} \Delta_{Fg}(\und{x},\und{x}').
\]
Given $\F=\{F_n\}_{n=1}^\infty$ we define the \emph{(right) Weyl pseudometric} on $X^G$ as
\[D_W\big(\und{x},\und{z}\big)=\limsup_{n\to\infty}\frac{1}{|F_n|}\sup_{g\in G}\sum_{f\in F_ng}\rho(x_f,z_f)=\limsup\limits_{n\to\infty}\frac{1}{|F_n|}\Delta_{F_n}^{*}(\underline x,\underline z).\]

\begin{rem}
It follows from Lemma~\ref{lS} below that $D_W(x,x')$ does not depend on the choice of $\mathcal F$ and by Corollary~\ref{limitexists} the upper limit above is a limit.
\end{rem}

The Weyl pseudometric induces a pseudometric on $(X,G)$, namely given $x,z\in X$ we set $\DW(x,z)=D_W(\und{x}_G,\und{z}_G)$. The convergence in $D_W$ on $(X,G)$ is called the \emph{quasi-uniform  convergence}. 

For a yet another characterization of Weyl pseudometric we need some machinery from \cite{DFR}.

A \emph{$k$-cover} $(k\in\N)$ of a set $F\in \Fin(G)$ is a tuple $(K_1,\ldots,K_r)$ of elements of $\Fin(G)$ such that for each $g\in F$ the set $\{ 1 \leq i \leq r \colon g\in K_i\}$ has at least $k$ elements.
A function $H\colon\Fin(G)\cup\{\emptyset\} \to [0,\infty)$:
\begin{itemize}[leftmargin=*]
\item \emph{satisfies Shearer's inequality} (see \cite{DFR}) if for any $F\in \Fin(G)$ and any $k$-cover $(K_1,\ldots,K_r)$ of $F$, we have
$H(F) \leq (1/ k) ( H(K_1)+\ldots + H(K_r))$,
\item  is \emph{$G$-invariant} if $H(Fg) = H(F)$ for every $g\in G$ and $F\in \Fin(G)$,
\item is \emph{monotone} if for all $A, B\in \Fin(G)$ with $A\subset B$ one has $H(A)\leq H(B)$,
\item is \emph{subadditive} if for all $A, B\in \Fin(G)$ with $A\cap B=\emptyset$ one has \[H(A\cup B)\leq H(A)+H(B).\]
\end{itemize}

\begin{lem}\label{lS}
If $\und{x},\und{x}' \in X^G$ and the function $H\colon \Fin(G) \to [0,\infty)$ is given by \newline
$H(F) = \Delta_F^*(\und{x},\und{x}')$,  then $H$ is $G$-invariant, satisfies Shearer's inequality and obeys the \emph{infimum rule}, that is, for each F{\o}lner sequence $\mathcal F=\{F_n\}_{n\in\N}$ one has
\[
D_W\big(\und{x},\und{x'}\big)=\limsup_{n\to\infty} H(F_n)/|F_n| = \inf_{F\in \Fin(G)} H(F)/|F|.
\]
\end{lem}

\begin{proof}
It is obvious that $H$ is $G$-invariant. We will show that it satisfies Shearer's inequality.
Let $F\in\Fin(G)$ and let $(K_1,\ldots,K_r)$ be a $k-$cover of $F$. Because every element of $F$ belongs to $K_i$ for at least $k$ indices $1\le i \le r$, we have
\begin{align*}
{1\over k}(H(K_1)+\ldots +H(K_r)) &\geq {1\over k} \sup_{g\in G}(\Delta_{K_1g}(\und{x},\und{x}')+\ldots+\Delta_{K_rg}(\und{x},\und{x}'))
\\
&\geq   {1\over k}\sup_{g\in G}(k \Delta_{Fg}(\und{x},\und{x}')) 
= H(F). 
\end{align*}
By \cite[Proposition 3.3]{DFR} every $G$-invariant, non-negative function on $\Fin(G)$ that satisfies Sharer's inequality, obeys the infimum rule.
\end{proof}

We will need the following theorem which comes from \cite[Theorem 6.1]{LW}.
\begin{thm}\label{LW}
If $H\colon\Fin(G)\to[0,\infty)$ is $G$-invariant, monotone, subadditive and such that $H(\emptyset)=0$, then for any \Folner sequence $\mathcal F=\{F_n\}_{n\in\N}$ the expression $H(F_n)/|F_n|$ converges as $n\to\infty$ and is independent of $\mathcal F$.
\end{thm}

\begin{cor}\label{limitexists}
The upper limit in the definition of the Weyl pseudometric is a limit.
\end{cor}

The proof of the next lemma follows the same lines as \cite[Lemma 3.3]{BBF}.

\begin{lem}\label{ano}
For any $\und{x},\und{x}'\in X^G$ one has
\[
D_W\big(\und{x},\und{x'}\big)=\sup_{\F}D_{B,\mathcal F}(\und x, \und{x'})=\sup_{\F}\limsup_{n\to \infty} {1\over | F_n|} \sum_{f\in F_n} \rho(x_f,x'_f).
\]
\end{lem}

\begin{proof}
Note that for every \Folner sequence $\{F_n\}_{n\in\N}$ and $\{g_n\}_{n\in\N}\subset G$ the sequence $\{F_ng_n\}_{n\in\N}$ is also F{\o}lner. Therefore the right hand side of the requested equality is greater than or equal to the left one.

By Lemma~\ref{lS} to prove the converse inequality it is enough to show that for every finite set $K\subset G$ there exists $t\in G$ such that for every $\beta\in(0,1)$ satisfying
\[ \alpha:=\sup_{\F}\limsup_{n\to \infty} {1\over | F_n|} \sum_{f\in F_n} \rho(x_f,x'_f)>\beta\text{ one has }
{1\over | K|} \sum_{f\in Kt}\rho(x_f, x'_f)>\beta. \] To this end choose a \Folner sequence $\{F_n\}_{n\in\N}$ and $\{k_n\}_{n\in\N}\nearrow\infty$ satisfying \[{1\over | F_{k_n}|} \sum_{f\in F_{k_n}} \rho(x_f,x'_f) >\frac{\beta+\alpha}{2}\text{ for every }n\in\N.\]  Since $K$ is finite, there exists $n\geq N$ such that for every $k\in K$ one has
\[
|F_n\setminus kF_n|/|F_n|\leq |kF_n\triangle F_n|/|F_n| < (\alpha-\beta)/2.
\]
Consequently, for $n$ large enough and any $k\in K$ one has
\[
{1\over |F_{k_n}|} \sum_{f\in kF_{k_n}} \rho(x_f,x'_f) \geq {1\over | F_{k_n}|} \sum_{f\in F_{k_n}} \rho(x_f,x'_f) - {1\over |F_{k_n}|} \sum_{f\in F_{k_n}\setminus kF_{k_n}} \rho(x_f,x'_f) > \beta,
\]
which leads to
\[
\sum_{t\in F_{k_n}} \sum_{f\in Kt}\rho(x_f,x'_f) = \sum_{k\in K} \sum_{f\in kF_{k_n}}\rho(x_f,x'_f) > | K| | F_{k_n}| \beta.
\]
Hence there exists $t\in F_{k_n}$ with the required property and the claim follows.
\end{proof}

Define $D_W'\colon X^G\times X^G\to\mathbb R_+$ as
\[D_W'\big(\und{x},\und{z}\big)=\inf\Big\{\eps>0\,:\,D^*\big(\big\{g\in G\,:\,\rho(x_g, z_g)>\eps\big\}\big)<\eps\Big\}.
\]
Lemma \ref{lem:db-prim} yields immediately the following corollaries.
\begin{cor}\label{cor:eqi}
The pseudometrics $D_W$ and $D_W'$ are uniformly equivalent.
\end{cor}
\begin{cor}\label{cor:equiv_m}
Let $\tilde\rho$ be another compatible metric on $X$ and $\tilde D_{W}$ be defined as $D_{W}$ above with $\tilde\rho$  in place of $\rho$. Then
$D_{W}$ and $\tilde D_{W}$ are uniformly equivalent on $X^G$.
\end{cor}

It occurs that if $(\alf^G,G)$ is a shift space,
then the Besicovitch pseudometric is uniformly equivalent with the $\dbar$-pseudometric
measuring the upper density of the set of coordinates at which two symbolic sequences differ. As a corollary we get an analogous claim for the Weyl pseudometric and the upper Banach density.
We omit the proof of Theorem~\ref{thm:familyF} as it follows the same lines as \cite[Proposition 1, Corollary 1]{DI} (see also \cite[Theorem 4]{KLOg}).

Given a continuous function $f\colon X\to\R$ define
\[
f(\und{x}_G)=\{f(gx)\}_{g\in G}\in \R^G.
\]
For a family $\mathcal{K}$ of continuous functions from $X$ to $\mathbb{R}$ and a \Folner sequence $\mathcal F$ set
\[
D^{\mathcal{K}}_W(x,x')=\sup_{\phi\in\mathcal{K}} D_W(\phi(\und{x}_G),\phi(\und{x}'_G)),\quad D^{\mathcal{K}}_{B,\mathcal F}(x,x')=\sup_{\phi\in\mathcal{K}} D_{B,\mathcal F}(\phi(\und{x}_G),\phi(\und{x}'_G)).
\]
\begin{thm}\label{thm:familyF}
Let $(X,G)$ be a dynamical system and let $\mathcal{K}$ be a uniformly equicontinuous and uniformly bounded family of real-valued functions on $X$ such that
$\mathcal{K}_G=\{x\mapsto\phi(gx): \phi\in\mathcal{K},\, g\in G\}$ separates the points of $X$. Then for any \Folner sequence $\mathcal F$ the pseudometrics $D_{B,\mathcal F}$ and $D^\mathcal{K}_{B,\mathcal F}$ are uniformly equivalent on $X$. Moreover, the modulus of uniform equivalence does not depend on the choice of $\mathcal F$ and consequently, the pseudometrics $D_W$ and $D^{\mathcal K}_W$ are also uniformly equivalent.
\end{thm}
Observe that any finite family $\mathcal{K}$ is uniformly equicontinuous and uniformly bounded.
In particular, taking $\mathcal{K}=\{\iota_e\}$, where $\iota_e(\und x_G)=x_e$ we obtain:
\begin{cor}\label{cor:uniformlyequiv}


If we consider the dynamical system $(\alf^G,G)$, then the Besicovitch pseudometric and
\emph{$\bar{d}_{\mathcal F}$-pseudometric} given by
\[
\bar{d}_{\mathcal F}(x,x')=\bar{d}_{\mathcal F}(\{g\in G : x_g\neq x'_g\})=\limsup_{n\to\infty}\big|\{f\in F_n : x_f\neq x'_f\}\big|/|F_n|,
\]
are uniformly equivalent and a modulus of uniform equivalence does not depend on $\mathcal F$.
The same is true for the Weyl pseudometric: $D_W$ is uniformly equivalent to the pseudometric $D^*$ given for $\und x, \und x'\in \alf^G$ by
\[
D^*(\und x,\und x')=D^*(\{g\in G : x_g\neq x'_g\})=\lim_{n\to\infty}\sup_{g\in G}\big|\{f\in F_ng : x_f\neq x'_f\}\big|/|F_n|.
\]
\end{cor}

To sum up this section we gather the facts about $D_W$ in one theorem.
\begin{thm}\label{tog}
Let $G$ be a countable amenable group acting on a compact metric space $(X,\rho)$. Fix any \Folner sequence $\{H_n\}_{n\in\N}$. Then for any $\und x, \und z\in X^G$ one has
\begin{multline*}D_W(\und x, \und z)=\lim_{n\to\infty}\sup_{g\in G}\frac{1}{|H_n|}\sum_{f\in H_ng}\rho(x_f,z_f)=\sup_{\mathcal F}D_{B,\mathcal F}(\und x, \und z)=\\=\sup_{\mathcal F}\limsup\limits_{n\to\infty}\frac{1}{|F_n|}\sum_{f\in F_N}\rho(x_{f},z_{f})=\inf_{F\in\Fin(G)}\frac{1}{|F|}\sup_{g\in G}\sum_{f\in F}\rho(x_{fg}, z_{fg}).\end{multline*}
Moreover, $D_W$ is uniformly equivalent to $D_W'$ given by \[D_W'\big(\und{x},\und{z}\big)=\inf\left\{\eps>0\,:\,\lim_{N\to\infty}\sup_{g\in G}\frac{1}{|F_N|}|\{f\in F_Ng\,:\,\rho(fx, fz)>\eps\}|<\eps\right\}.\]
If $X=\alf^G$ and $G$ acts by the shift, then $D_W$ and $D_W'$ are uniformly equivalent to the pseudmetric
\[
D^*(\und x,\und z)=D^*(\{g\in G : x_g\neq z_g\}).
\]
\end{thm}

\section{Quasi-uniform Convergence and Simplices of Measures}

Given $F\in\Fin( G)$ and $\und x=\{x_g\}_{g\in G}\in X^G$ we define the \emph{empirical measure} of $x$ with respect to $F$ as
$$\m(\und x, F)=\frac{1}{|F|}\sum_{g\in F} \hat{\delta}_{x_g}\in\M(X).$$

Fix 
$\mathcal F=\{F_n\}_{n\in\N}$. 
A measure $\mu\in \M(X)$ is a \emph{distribution measure} for $\{x_g\}_{g\in G}\in X^G$ if $\mu\in\M(X)$ is a limit of a subsequence of
$\{\Emp(\underline{x},F_n)\}_{n=1}^\infty$. The set of all distribution measures of $\underline{x}$ (along $\mathcal F$) is denoted by $\om{\underline{x}}$. Clearly $\om{\underline{x}}$ is a closed and nonempty subset of $\M(X)$. For $x\in X$ we put $\om{x}=\om{\und x_G}$.

\begin{rem}
If $\F$ satisfies $|F_n|/|F_{n+1}|\to 1\text{ as }n\to\infty$ and $F_n\subset F_{n+1}$ for every $n\in\N$, then for every $\und{x}\in X^G$ the set $\om{\und{x}}$ is connected. 
 \end{rem}

\begin{proof}
 Fix $\und{x}\in X^G$. Let $B$ be a Borel set.
 One has \begin{multline*}\m(\und{x},F_{n+1})(B)
 =\frac{|F_{n}|}{|F_{n+1}|}\cdot\frac{1}{|F_{n}|}\sum_{f\in F_{n+1}}\hat{\delta}_{x_f}(B)\leq \m(\und{x}, F_{n})(B)+ \frac{|F_{n+1}|-|F_n|}{|F_{n+1}|}.
 \end{multline*}
Hence $\di(\m(\und{x}, F_{n+1}),\m(\und{x},F_{n}))\leq (|F_{n+1}|-|F_n|)/|F_{n+1}|\to 0\text{ as }n\to\infty$
\newline
and so $\om{\und{x}}$ is connected by~\cite[Theorem 1]{Schaefer}.
 \end{proof}
 
\begin{prop}
If $\eps>0$ and $\mathcal F=\{F_n\}_{n\in\N}$ is a F{\o}lner sequence with \[\lim\limits_{n\to\infty}|F_n|/|F_{n+1}|=1-\eps \text{  and  } F_n\subset F_{n+1},\]
then there are a dynamical system $(X,G)$ and $x\in X$ such that $\hat\omega_{\mathcal F}(x)$ is not connected.
\end{prop}

\begin{proof}
Let $X=\{0,1\}^{G}$ with the shift action of $G$. Let $x=\{x_g\}_{g\in G}$ satisfy $x_g=1$ if $g\in F_1$ or $g\in F_{2n+1}\setminus F_{2n}$ for some $n\in\N$ and $x_g=0$ otherwise.  
Define \[\mathcal E=\{\mu\in\om{x}\,:\,\mu=\lim\limits_{n\to\infty}\m(x,F_{2k_n})\text{ for some }k_n\nearrow\infty\},\]\[\mathcal O=\{\nu\in\om{x}\,:\,\nu=\lim\limits_{n\to\infty}\m(x,F_{2l_n+1})\text{ for some }l_n\nearrow\infty\}.\]
Clearly, $\mathcal O\cup\mathcal E=\om{x}$.
Moreover, if $\mu\in\mathcal E$, then for $\eta<1/2$ one has\footnote{Recall that by $[1]^{\eta}$ we mean the set of all points within $\eta$ of the cylinder $[1]$.} $\mu([1]^{\eta})=\mu([1])=(1-\eps)/(2-\eps)$ and if $\nu\in\mathcal O$, then $\nu([1])=1/(2-\eps)$ (we leave the computations to the reader). Therefore $\dist_{\di}(\mathcal O,\mathcal E)>0$ and so $\mathcal O\cap\mathcal E=\emptyset$. Since both $\mathcal O$ and $\mathcal E$ are non-empty and closed, $\om{x}$ is not connected.
 \end{proof}



For $\und x\in X^G$ let \[\M_G(\und x)=\bigcup_{\mathcal F}\om{\und x}.\]
Denote also $\M_G(x)=\M_G(\underline x_G)$.
\begin{rem}\label{domk}
For any $x\in X$ the set $\M_G(x)$ is closed.
\end{rem} 

\begin{proof}
Fix $x\in X$ and pick $\{\mu_k\}_{k\in\N}\subset\M_G(x)$ such that $\mu_k\to\mu$ as $k\to\infty$ for some $\mu\in\M(X)$. We will show that $\mu\in\M_G(x)$. For every $k\in\N$ let $\mathcal F^{(k)}=\{F^{(k)}_n\}_{n\in\N}$ be a \Folner sequence such that $\m(x,F^{(k)}_n)\to\mu_k$ as $n\to\infty$. Enumerate elements of $G$ as $g_1,g_2,\ldots$. For every $k\in\N$ choose $n(k)\in\N$ such that $\di\left(\m(x,F^{(k)}_{n(k)}),\mu_k\right)<1/k$ and for every $i\leq k$ one has $|g_iF_{n(k)}^{(k)}\triangle F_{n(k)}^{(k)}|/|F_{n(k)}^{(k)}|<1/k$.
Define $\mathcal F=\{F_{n(k)}^{(k)}\}_{k\in\N}$. Then $\mathcal F$ is a \Folner sequence and $\om{x}=\{\mu\}$. 
\end{proof}

Downarowicz and Iwanik proved that if $x_n\to x$ quasi-uniformly, then $\M_{\mathbb Z}(x_n)$ tends to $\M_{\mathbb Z}(x)$ in $\HD$ \cite[Theorem 2]{DI}. We generalize this result.

\begin{rem}
In \cite{BFK} it was proved that, in general, the topology generated by the Weyl pseudometric is neither separable, nor complete. In particular, it is not compact.
\end{rem}

Repeating the proof of \cite[Theorem 7]{KLOg} we get:

\begin{thm}\label{HD}
For every $\eps>0$ there is $\delta>0$ such that if $\underline{x},\underline{x}'\in X^G$ and $D_{B,\mathcal F}(\underline x,
\underline x')<\delta$ then $\HD(\om{\underline x},\om{\underline x'})<\eps$. Moreover, $\delta$ is independent of $\mathcal F$.
\end{thm}


\begin{cor}\label{cor:D_B-zero}If $\und{x},\und{x}'\in X^G$ and $D_{B,\mathcal F}(\underline x,\underline x')=0$, then $\om{\underline x}=\om{\underline x'}$ for every $\F$. In particular, if $x,z\in X$, $x$ is generic for $\mu\in\M_G(X)$ along $\F$ and $D_{B,\mathcal F}(x,z)=0$, then also $z$ is generic for $\mu$ along $\mathcal F$.
\end{cor}

Theorem~\ref{w} is a straightforward consequence of Theorem~\ref{HD} and Lemma~\ref{ano}.
\begin{thm}\label{w}
The function $(X,D_W)\ni x\to \M_G(x)\in(2^{\M(X)},\HD)$ is uniformly continuous.
\end{thm}

To prove Lemma~\ref{Mx} we need two technical lemmas. We omit their easy proofs.
\begin{lem}\label{szac2}
Let $F\in\Fin(G)$ and $x,y\in X$. If for every $f\in F$ one has $\rho(fx, fy)\leq~\eps$, then $\di(\m(x, F), \m(y,F))\leq\eps$.
\end{lem}

\begin{lem}\label{szac}
Let $\alpha_1,\ldots,\alpha_k, \beta_1,\ldots, \beta_k \in[0,1]$ be  such that $\sum_{i=1}^k\alpha_i=\sum_{i=1}^k\beta_i=1$.
Then for all $\mu_1,\ldots,\mu_k,\nu_1,\ldots,\nu_k\in\mathcal M(X)$ one has 
$$\di\left(\sum_{i=1}^k\alpha_i\mu_i,\sum_{i=1}^k\beta_i\nu_i\right)\leq\frac{1}{2}\sum_{i=1}^k|\alpha_i-\beta_i|+\max\big\{\di(\mu_i,\nu_i)\,:\,1\leq i\leq k\big\}.$$
\end{lem}

\begin{lem}\label{Mx}
For every $x\in X$ one has
$\mathcal{M}_G(x)=\mathcal{M}_G(\overline{Gx})$.
\end{lem}

\begin{proof}
It is obvious that $\mathcal{M}_G(x)\subset\mathcal{M}_G(\overline{Gx})$. 
Fix $\mu\in \mathcal{M}_G(\overline{Gx})$ and a tempered \Folner sequence $\mathcal F=\{F_n\}_{n\in\N}$.
From the Krein-Milman theorem we get \[\mu=\lim\limits_{n\to\infty}(p_1^{(n)}/q_1^{(n)})\nu_1^{(n)}+\ldots+(p_n^{(n)}/q_n^{(n)})\nu_n^{(n)},\] where for every $n\in\N$ and $1\leq j\leq n$ one has $p_{j}^{(n)},q_{j}^{(n)}\in\N$, $\sum_{j=1}^{n}p_{j}^{(n)}/q_{j}^{(n)} = 1$ and $\nu_{j}^{(n)}\in\M^e_G(\overline{Gx})$.
Fix $n\in\N$. It follows from Remark~\ref{domk} that it is enough to prove that for every $\eps>0$ there exists a \Folner sequence $\mathcal F_{\eps}=\{F_{k,\eps}\}_{k\in\N}$ such that for every $k\in\N$ large enough one has
\[\di\left(\frac{p_1^{(n)}}{q_1^{(n)}}\nu_1^{(n)}+\ldots+\frac{p_n^{(n)}}{q_n^{(n)}}\nu_n^{(n)},\, \m\left(x, F_{k,\eps}\right)\right)<\eps.\]
For $j=1\ldots n$ let $x_j$ be a regular generic point for $\nu_j^{(n)}$ along $\mathcal F$. Pick $K\in\N$ such that if $k\geq K$ and $1\leq i\leq n$ then $\di\left(\nu^{(n)}_i, \m(x_i, F_k)\right)<\eps/2.$
Fix $k\geq K$. Let $\delta>0$ be such that if $f\in F_k$, $a,b\in X$ and $\rho(a,b)<\delta$ then $\rho(fa, fb)<\eps/2$.
Denote $L_j=p_j^{(n)}\cdot(q_1^{(n)}\cdot\ldots\cdot q_n^{(n)})/q_j^{(n)}.$ Using Lemma~\ref{szac2} and the fact that for every $1\leq j\leq n$ there exists infinitely many $g\in G$ such that $\rho(gx, x_j)<\delta$, one can find $g_{1}^{(1)},\ldots, g_{L_1}^{(1)},\ldots, g_{1}^{(n)},\ldots, g_{L_n}^{(n)}\in G$ such that $F_k g_{i}^{(j)}\cap F_kg_{i'}^{(j')}=\emptyset$ for $(i,j)\neq(i',j')$
and for every $1\leq j\leq n$ and $1\leq i\leq L_j$ one has\begin{multline*}
\di\left(\m(x, F_kg_{i}^{(j)}),\nu_j^{(n)}\right)\leq\di\left(\m(x, F_kg_{i}^{(j)}),\m(x_j, F_k)\right)+\di\left(\m(x_j, F_k),\nu_j^{(n)}\right)\leq\eps.\end{multline*}
We set \[F_{k,\eps}=\bigsqcup_{j=1}^{n}\bigsqcup_{i=1}^{L_j}F_kg_{i}^{(j)}.\]
Clearly $\mathcal F_{\eps}=\{F_{k,\eps}\}_{k\in\N}$ is a \Folner sequence as the number of shifted copies of \Folner sets which form $\mathcal F_{k,\eps}$ does not depend on $k$.
By Lemma~\ref{szac} we get
\begin{multline*}\di\left(\m(x,F_{k,\eps}),\, (p_1^{(n)}/q_1^{(n)})\nu_1^{(n)}+\ldots+(p_n^{(n)}/q_n^{(n)})\nu_n^{(n)}\right)=\\=\di\left(\sum_{j=1}^n\sum_{i=1}^{L_j}(\prod_{s=1}^nq_s^{(n)})^{-1}\cdot\m(x, F_kg_{i}^{(j)}),\, \sum_{j=1}^n\sum_{i=1}^{L_j}(\prod_{s=1}^nq_s^{(n)})^{-1}\cdot\nu^{(n)}_j\right)\leq \eps.\qedhere
\end{multline*}
\end{proof}

As a corollary of Theorem~\ref{w} and Lemma~\ref{Mx} we get the following.
\begin{cor}\label{thm:ciagwe'}
The function $(X,D_W)\ni x\to \M_G(\overline{Gx})\in(2^{\M(X)},\HD)$ is uniformly continuous.
\end{cor}
It follows from Corollary~\ref{thm:ciagwe'} that:
\begin{cor}
The number of ergodic invariant measures supported at $\overline{Gx}$ is lower semicontinuous with respect to $D_W$.
\end{cor}
\begin{proof}
It is enough to show that the number of extreme points
of a compact convex subset of a locally compact space varies lower semicontinuously with respect to $\HD$. To see this assume that $K_n$ for $n\in\N$ are compact and convex sets such that $H(K_n,K_0)\to 0$ as $n\to\infty$. For every $n\in\N\cup\{0\}$ denote by $s_n$ the number of extreme points of $K_n$. We can assume that $\liminf\limits_{n\to\infty}s_n=m$ for some $m\in\N$ as the claim in the opposite case is obviously true. Let $(l_n)_{n\in N}$ be such that $\liminf\limits_{n\to\infty}s_{l_n}=m$. For every $n\in\N$ let $S^n=(s_1^n,\ldots, s_m^n)$ be a tuple such that $\text{conv}S^n=K_{l_n}$. We can assume that for every $1\leq i\leq m$ one has $s_i^n\to s_i$ for some $s_i\in K_0$. We will show that $\text{conv}\{s_1,\ldots, s_m\}=K_0$. To this end fix $x\in K_0$. For $1\leq i\leq m$ and $n\in\N$ pick $\alpha^{(n)}_m\in[0,1]$ such that:
\begin{itemize}
\item for every $n\in\N$ one has $\alpha_1^{(n)}+\ldots+\alpha_m^{(n)}=1$,
\item $\alpha_1^{(n)}s_1^{n}+\ldots+\alpha_m^{(n)}s_m^{n}\to x$ as $n\to\infty$.
\end{itemize}
Without lost of generality for every $1\leq i\leq m$ let $\alpha_i^{(n)}\to\alpha_i$ for some $\alpha_i\in[0,1]$. Then
$\alpha_1+\ldots+\alpha_m=1$ and $\alpha_1s_1+\ldots+\alpha_ms_m=x$.
This finishes the proof.
\end{proof}

\section{Topological Entropy}
This section generalizes \cite[Section 5]{DI}.
For an open cover $\mathcal U$ of the space $X$ let $\mathcal N(\mathcal U)$ be the minimal cardinality of a subcover of $\mathcal U$. Define $\mathcal U\vee\mathcal V=\{U\cap V\colon U\in\mathcal U, V \in\mathcal V\}$ and note that $\vee$ is an associative operation. Given $F=\{f_1,\ldots,f_s\}\subseteq G$ and $f\in F$ we write $f^{-1}\mathcal U=\{f^{-1}U\colon U\in \mathcal U\}$ and $\mathcal U^F = (f_1^{-1}\mathcal{U})\vee\ldots\vee (f_s^{-1}\mathcal{U})$. By Theorem~\ref{LW}, the sequence $\log\mathcal N(\mathcal U^{F_n})/|F_n|$ has a limit $h(X, G, \mathcal U)$ which is independent of $\F$.
The \emph{topological entropy} of $(X,G)$ is defined as 
\[
\htop(X)=\htop(X,G)=\sup\{h(X,G,\mathcal U)\,:\,\mathcal U\text{ is an open cover of }X\}.
\]

Let $F\in\Fin(G)$  and $\delta, \eps\in(0,1]$. Let $\rho_F(x,y)= \max\{\rho(gx,gy)\colon g\in F\}$ for $x,y\in X$. A set $Z\subset X$ is called
\begin{itemize}[leftmargin=*]
\item \emph{$(F,\eps)$-separated} if $\rho_F(x,z)>\eps$ for $x,z\in Z$ with $x\neq z$,
\item \emph{$(F,\eps)$-spanning} if for each $x\in X$ there is $z\in Z$ such that $\rho_F(x,z)\leq\eps$,
\item \emph{$(F,\eps,\delta)$-spanning} if for every $x\in X$ there exists $z\in Z$ such that 
\[|\{g\in F\,:\,\rho(gx, gz)\leq\eps\}|>(1-\delta)\cdot|F|,\]
\item \emph{$(F,\eps,\delta)$-separated} if for every $x,z\in Z$ either $x=z$ or
\[
|\{g\in F\,:\,\rho(gx,gz)>\eps\}|>\delta\cdot|F|.
\]
\end{itemize}
Fix a \Folner sequence $\mathcal F=\{F_n\}_{n\in\N}$.
Denote by $s(n,\eps)$, $r(n,\eps)$, $\tilde{s}(n,\eps,\delta)$, $\tilde{r}(n,\eps, \delta)$ the maximal cardinality of an $(F_n, \eps)$-separated set, minimal cardinality of an $(F_n, \eps)$-spanning set, maximal cardinality of an $(F_n, \eps, \delta)$-separated set, and the minimal cardinality of an $(F_n, \eps, \delta)$-spanning set, respectively.
Define
\begin{align*}
h_s(X,G)&=\lim\limits_{\eps\searrow 0^+}\limsup\limits_{n\to\infty}\frac{\log s(n,\eps)}{|F_n|},& \tilde{h}_s(X,G)&=\lim\limits_{\delta\searrow 0^+}\lim\limits_{\eps\searrow 0^+}\limsup\limits_{n\to\infty}\frac{\log \tilde s(n,\eps, \delta)}{|F_n|},\\
h_r(X,G)&=\lim\limits_{\eps\searrow 0^+}\limsup\limits_{n\to\infty}\frac{\log r(n,\eps)}{|F_n|},&\tilde{h}_r(X,G)&=\lim\limits_{\delta\searrow 0^+}\lim\limits_{\eps\searrow 0^+}\limsup\limits_{n\to\infty}\frac{\log \tilde r(n,\eps, \delta)}{|F_n|}.
\end{align*}
 It is known that $h_s(X,G)=h_r(X,G)=\htop(X,G)$ and we will show that $\tilde{h}_s(X,G)=\tilde{h}_r(X,G)=\htop(X,G)$.

\begin{lem}[\cite{Shields}, Lemma I.5.4]\label{suma}
Let $E_S(\eps)=-\eps\log \eps - (1-\eps)\log (1-\eps)$. If $0<\eps\le 1/2$, then
\[
\sum_{j=0}^{\lfloor n \eps \rfloor}{\binom{n}{j}}\le 2^{n\cdot E_S(\eps)} \quad\text{for $n\ge 1$}.
\]
\end{lem}
Lemma~\ref{lm1} is an analog of \cite[Lemma 3]{DI}.
\begin{lem}\label{lm1}
For any finite open cover $\mathcal U$ of $X$ one has $h(\mathcal U, G)\leq \tilde{h}_r(X,G)$.
\end{lem}

\begin{proof}
Let $2\eps$ be the Lebesgue number of $\mathcal U$. Fix $\delta\in(0,1/2)$.
We will show that \begin{equation}\label{e1}
h(X,\mathcal U, G)\leq \limsup\limits_{n\to\infty}\frac{1}{|F_n|}\log r(n,\eps,\delta)+\log 2\cdot E_S(\delta)+\delta.
\end{equation}
Fix $n\in\N$. Let $Z$ be a $(F_n,\eps, \delta)$-spanning set such that $|Z|=r(n,\eps,\delta)$.
For every $K\subset F_n$ and $y\in Z$ define \[V(y,K)=\bigcap_{g\in K}g^{-1}\big(\{a\in X\,:\,\rho(gy,a)<\eps\}),\]
Notice that $V(y,K)\subset \bigvee_{g\in K}g^{-1}\mathcal U$.
Therefore \[\mathcal W=\bigcup_{K\subset F_n,\\ |K|>|F_n|(1-\delta)}\left(\{V(y,K)\,:\,y\in Z\}\vee\bigvee_{g\in F_n\setminus K}g^{-1}\mathcal U\right)\]
is a refinement of a cover $\mathcal U^{F_n}$. Hence
\begin{multline*}
\mathcal N\left(\mathcal U^{F_n}\right)\leq\mathcal N(\mathcal W)\leq \big|\big\{K\subset F_n\,:\,|K|>|F_n|\cdot(1-\delta)\big\}\big|\cdot |Z|\cdot |\mathcal U|^{|F_n|\cdot\delta}=\\
=\sum_{k=0}^{\lceil |F_n|\delta\rceil}{| F_n| \choose k}\cdot |Z|\cdot |\mathcal U|^{|F_n|\cdot\delta}
\leq 2^{|F_n|E_S(\delta)}\cdot r(n,\eps,\delta)\cdot |\mathcal U|^{|F_n|\cdot\delta}.
\end{multline*}
This implies (\ref{e1}). To finish the proof note that $E_S(\delta)\to 0$ as $\delta\to 0$.
\end{proof}

\begin{thm}
One has $\tilde{h}_s(X,G)=\tilde{h}_r(X,G)=\htop(X,G)$.
\end{thm}
\begin{proof}
Note that for all $n\in\N$, $\eps>0$ and $\delta\in(0,1]$ one has
$r(n,\eps,\delta)\leq s(n,\eps,\delta)\leq r(n,2\eps, \delta)$ and so $\tilde{h}_s(X,G)=\tilde h_r(X,G)$.
Using Lemma~\ref{lm1} we get $\htop(X,G)\leq \tilde h_r(X,G)=\tilde h_s(X,G)\leq h_s(X,G)=\htop(X,G)$ and the claim follows.
\end{proof}

For any $x\in X$ define $h(x)=\htop(\overline{Gx},G)$.
We will show that $h(x)$ vary lower semicontinuously. 

For $x\in X$, $n\in\N$, $\eps>0$ and $\delta\in(0,1]$ let $s(x,n,\eps,\delta)$ be the maximal cardinality of $(F_n, \eps, \delta)$-separated set for $(\overline{Gx},G)$.
Lemma~\ref{l3} is proved in the same way as \cite[Lemma 4]{DI}.
\begin{lem}\label{l3}
If $D^*(\{g\in G\,:\, \rho(gx, gy)>\eps\})<\delta$, then $s(y,n,\eps, \delta)\geq s(x, n, 3\eps, 3\delta)$ for all $n$ big enough.
\end{lem}


Theorem~\ref{thm:poll} is an analog of \cite[Theorem 3]{DI}.

\begin{thm}\label{thm:poll}
The function $X\ni x\mapsto h(x)\in \mathbb R_+\cap\{\infty\}$ is lower semicontinuous with respect to $D_W$ and $D_W'$.
\end{thm}

\begin{proof}
By Corollary~\ref{cor:eqi} it is enough to consider only $D_W'$.
Fix $x\in X$ such that $h(x)<\infty$ and $\eta>0$. There exists $\delta,\eps>0$ such that $$\lim\limits_{n\to\infty}s(x,n,\eps, \delta)>h(x)-\eta.$$ It follows from Lemma~\ref{l3} that for every $y\in X$ such that $D_W'(x,y)<\delta$ one has
$$h(y)\geq \lim\limits_{n\to\infty}s(y,n,\eps/3, \delta/3)>h(x)-\eta.$$
If $h(x)=\infty$ then the proof is similar.
\end{proof}

As a corollary of Theorem~\ref{thm:poll} and Corollary~\ref{cor:eqi} we get the following.

\begin{thm}\label{thm:pol}
The function $x\mapsto h(x)$ is lower semicontinuous with respect to $D_W$-pseudometric on $X$ and usual metric on $\mathbb R_+\cup\{\infty\}$.
\end{thm}

\begin{rem}\label{entropianiejestciagla}
In \cite[Example 3]{DI} it is shown that the function $(X,D_W)\ni x\mapsto h(x)\in \mathbb R_+\cup\{\infty\}$ need not to be continuous even in case of $\mathbb Z$ action. The example there can be generalized. Let $G$ be a countable amenable group and let $X=Y^G$ for some infinite compact metric space $Y$. Let $G$ act on $X$ by the shift. Pick $y\in Y$ and $\{y_n\}_{n\in\N}\subset Y\setminus\{y\}$ such that $y_n\to y$ as $n\to\infty$. Let $x_n\in\{y,y_n\}^G$ be a point such that $\overline{O(x_n)}=\{y,y_n\}^G$. Define $x=y^G$. Then $D_W(x,x_n)\to 0$ as $n\to\infty$, but $h(x)=0$ while for every $n\in\N$ one has $h(x_n)=\log 2$.
\end{rem}
The function $x\mapsto h(x)$ turns out to be continuous if $X=\alf^G$, where $\alf$ is a finite set. The proof is independent from the one of Theorem~\ref{thm:poll}. Note that in this case $h(x)\leq\log|\alf|$.
\begin{thm}\label{entropiaciagla}
The function $x\mapsto h(x)$ is $D^*$-continuous on $\alf^G$.
\end{thm}

\begin{proof} Fix a \Folner\ sequence $\F=\{F_n\}_{n=1}^\infty$ and $\eps>0$. Let $\mathcal B_{F_n}(x)$ denote the set of configurations of shape $F_n$ that appear in $x$. Then
\[
\htop(\overline{Gx})=\lim_{n\to\infty}\frac{\log|\lang_{F_n}(x)|}{|F_n|}.
\]
Pick $0<\delta<1/4$ so that $t<2\delta$ implies $E_S(t)<\eps$. Note that
$D_W(x,x')<\delta$ implies that for all $g\in G$ we have
$
|\{f\in F_n : x_{fg}\neq x'_{fg}\}|<2\delta |F_n|$
for all $n$ large enough.
By Lemma~\ref{suma} we have for all $n$ large enough that
\[
|\lang_{F_n}(x')|\le|\alf|^{2\delta\cdot|F_n|}\cdot 2^{E_S(2\delta)\cdot |F_n|} \cdot |\lang_{F_n}(x)|.
\]
Thus
\[
\htop(\overline{Gx'})=\limsup_{n\to\infty}\frac{\log|\lang_{F_n}(x')|}{|F_n|}\le 2\delta\log{|\alf|}+\log 2\cdot E_S(2\delta)+\htop(\overline{Gx}).
\]
Interchanging the roles of $x$ and $x'$ finishes the proof.
\end{proof}

\section{The Number of Minimal Components of $\overline{Gx}$}\label{sec:m}

Denote by $m(x)\in\N\cup\{\infty\}$ the number of minimal components of $\overline{Gx}$. Our aim is to prove that $m$ is lower semicontinuous with respect to $D_W$, which generalizes \cite[Theorem 1]{DI}.
For $x\in X$ and $A\subset X$ define the \emph{set of return times} $N(x,A)=\{g\in G\,:\,gx\in A\}.$
Recall that a set $Z\subset X$ is \emph{invariant} if $GZ=\{gz\,:\,g\in G,z\in Z\}\subset Z$.

\begin{lem}\label{lem:minimalthick}
If $Z\subset \overline{Gx}$ is an invariant (not necessarily closed) set, then for any $\eps>0$ the set $N(x, Z^{\eps})$ is thick.
\end{lem}

\begin{proof}
We have to show that for every finite $F\subset G$ one can find $\gamma\in N(x,Z^{\eps})$ such that for every $f\in F^{-1}$ one has $f\gamma x\in Z^{\eps}$. Choose $y\in Z$. Let $\delta\in(0,\eps/2)$ be such that for all $a,b\in X$ with $\rho(a,b)<\delta$ one has $\rho(fa, fb)<\eps$ for every $f\in F^{-1}$. There exists $\gamma\in G$ such that $\rho(\gamma x, y)<\delta$. This means that $\gamma\in N(x,Z^{\eps})$ and for every $f\in F^{-1}$ one has $\rho(f\gamma x,y)<\eps$. This implies that $f\gamma x\in Z^{\eps}$.
\end{proof}

The next Lemma combines \cite[Proposition 2.18]{BKKPL} and \cite[Lemma 1]{DI}.

\begin{lem}\label{lem:synd}
Let $(X,G)$ be a transitive dynamical system and $x\in X$ be such that $X=\overline{Gx}$. Then the following conditions are equivalent:
\begin{enumerate}[leftmargin=*]
\item One has $m(x)\leq m$.
\item There exist $p\leq m$ closed $G$-invariant sets $N_1,\ldots, N_p\subset X$ such that for every $z\in X$ there exists $1\leq i\leq p$ satisfying $N_i\subset \overline{Gz}$.
\item There exist $p\leq m$ closed $G$-invariant sets $P_1,\ldots, P_p\subset X$ such that for every open set $U\subset X$ intersecting $P_i$ for every $1\leq i\leq p$ the set $N(x, U)$ is syndetic.
\item There exist $m$ points $z_1,\ldots, z_m\in X$ such that for every $\eps>0$ the set $N(x, Z^{\eps})$ is syndetic, where $Z=\{z_1,\ldots, z_m\}$.\label{2}
\end{enumerate}
\end{lem}

\begin{proof}
To prove that $(1)\Rightarrow (2)$ let $N_1,\ldots, N_p$ be all minimal subsets of $X$. Note that for $z\in X$ the set $\overline{Gz}$ is nonempty, closed and invariant and hence contains $N_i$ for some $1\leq i\leq p$.
Thus $(2)$ holds.

To show that $(2)\Rightarrow (3)$ let $P_i=N_i$ for $1\leq i\leq p$, where $N_i$'s are as in $(2)$. Choose an open set $U$ such that $U\cap P_i\neq\emptyset$ for every $1\leq i\leq p$. Then there exists $g\in G$ such that $x\in V:=gU$. Since $g^{-1}N(x, U)\supset N(x, V)$ it is enough to show that $N(x, V)$ is syndetic. Suppose that it is not true. Let $\{F_n\}_{n\in\N}$ be an increasing sequence of finite sets whose union is equal to $G$. Then for every $n\in\N$ one can find $\gamma_n\in G$ such that for every $f\in F_n$ one has $f\gamma_n x\notin V$. Let $y\in X$ be a limit point of $\{\gamma_n x\}_{n\in\N}$. Then $\overline{Gy}\cap V=\emptyset$, which contradicts $(2)$.

Obviously $(3)\Rightarrow (4)$ and it remains to show that $(4)\Rightarrow(1)$.
Suppose that $m(x)>m$. Let $Z_1, \ldots Z_{m+1}\subset \overline{Gx}$ be disjoint minimal sets.
Choose $\eps>0$ such that for any $i\neq j$ one has $\text{dist}(Z_i,Z_j)>2\eps$. Notice that for any set $Z$ consisting of $m$ points $z_1,\ldots, z_m$ there exists $1\leq i_0\leq m$ such that  $Z^{\eps}$ is disjoint from $Z_{i_0}^{\eps}$. Therefore $N(x, Z^{\eps})\cap N(x, Z_{i_0}^{\eps})=\emptyset$. By Lemma~\ref{lem:minimalthick} the set $N(x,Z_{i_0}^{\eps})$ is thick. Hence $N(x,Z^{\eps})$ is not syndetic.
\end{proof}

\begin{thm}
The function $m(x)\colon (X,D_W)\to \mathbb N\cup\{\infty\}$ is lower semicontinuous.
\end{thm}

\begin{proof}
Fix $x\in X$. Let $m\in\N$ fulfill $m(x)>m$. Choose minimal sets $Z_1,\ldots, Z_{m+1}\subset \overline{Gx}$ and $\eps>0$ with $Z_i\cap Z_j=\emptyset$ and $\text{dist}(Z_i,Z_j)>2\eps$ for every $i\neq j$. Fix $y\in X$ with $D_W'(x,y)<\eps/6$, we will show that $m(y)>m$. Let $\{F_n\}_{n\in\N}$ be a F{\o}lner sequence.
There exists $N\in\N$ such that for every $g\in G$ one has \[\left|\left\{f\in F_Ng\,:\,\rho(fx, fy)>
\eps/6\right\}\right|<\eps|F_N|/6<|F_N|/2.\]
Let $\delta\in(0,\eps/3)$ be such that if $a,b\in X$, $g\in F_N$ and $\rho(a,b)<2\delta$, then $\rho(ga, gb)<\eps/3$. Let $Z(y)=\{z_1,\ldots, z_{m(y)}\}$ be constructed as in Lemma~\ref{lem:synd} for $y$. Fix $1\leq s\leq m+1$. Then $N(x, Z_s^{\delta})\cap N(y, Z(y)^{\delta})\neq\emptyset$. Let $g(s)\in G$ and $1\leq r(s)\leq m(y)$ be such that $g(s)x\in Z_s^{\delta}$ and $\rho(g(s)y,z_{r(s)})<\delta.$ Notice that for more than the half of elements $f\in F_N$ one has \begin{multline*}\text{dist}(fz_{r(s)}, Z_s)\leq \rho(fz_{r(s)},fg(s)y)+\rho(fg(s)x,fg(s)y)+\dist (fg(s)x, Z_s)\leq\eps.
\end{multline*}
Consequently, if $s\neq t$, then also $r(s)\neq r(t)$. Therefore $m(y)\geq m+1>m$. 
\end{proof}

\begin{rem}
Note that the example presented in Remark~\ref{entropianiejestciagla} shows that there is a $G$ action such that $m$ is not $D_W$-continuous. Such a system exists for every countable amenable group $G$. (cf. \cite[ Example 2]{DI})
\end{rem}

\section{Toeplitz Configurations}\label{section7}
\subsection{Basic Properties and Regularity}
Throughout this section we assume that $G$ is a countable residually finite group, not necessarily amenable.
Let $\alf$ be a finite discrete metric space.
We call $\und x\in \alf^G$ a \emph{Toeplitz configuration} if for every $g\in G$ there exists $H<_f G$ such that $\und x(Hg)=\und x(g)$.
For $H\subset G$ and $a\in \alf$ define:
\[\text{Per}_{H}(\und x)=\big\{g\in G\,:\, \und x(g)=\und x(\gamma g)\text{ for every }\gamma\in H\big\},\]
 \[\text{Per}_{H}(\und x,a)=\big\{g\in G\,:\, \und x(\gamma g)=a\text{ for every }\gamma\in H\big\}.\]
Notice that $\und x\in  \alf^G$ is a Toeplitz configuration if and only if \[\bigcup\{\text{Per}_{H}(\und x)\colon H <_fG\} =G.\]


Let $\{H_i\}_{i=0}^{\infty}$ be a sequence of subgroups of $G$ with finite index such that $H_1\supset H_2\supset\ldots$ (we do not assume here that $\bigcap H_i=\{e\}$). The inverse limit of the system  $\lim\limits_{\longleftarrow}(G/H_i,\pi_i)$, where  $\pi_i\colon G/H_{i+1}\to G/H_i$ are natural projections, is called the $G$-odometer and is denoted by $\hat G$. The odometer is \emph{exact} if all $H_i$ are normal. Every $G$-odometer is a compact metric space with a natural $G$ action.

\begin{rem}
Notice that $\hat G$ is not necessarily  equal to a profinite completion of $G$ as here we consider only subgroups from the fixed sequence $\{H_n\}_{n\in\N}$.
\end{rem}
A nested sequence of finite index subgroups $\{H_n\}_{n\in\N}$ is a \emph{skeleton} of $\und x\in X^G$ if $\bigcup\text{Per}_{H_n}(\und x)=G$ and for every $n\in\N$ one has $\text{Per}_{H_n}(\und x)\neq\emptyset$ and if $g\in G$ is such that for every $a\in\alf$ one has  $\text{Per}_{H_n}(\und x,a)=\text{Per}_{H_n}(g\und x,a)$, then $g\in H_n$.

\begin{thm}[\cite{CortezPetite}, Corollary 6]
Every Toeplitz configuration has a skeleton.
\end{thm}

Our goal now is to define a \emph{regular} Toeplitz configuration. To justify that the regularity does not depend on the choice of a skeleton we need the following results from~\cite{CortezPetite}. Note that the standing assumption in \cite{CortezPetite} is that $G$ is finitely generated, but one can check that the results we need hold without the assumption from~\cite{CortezPetite}.

\begin{thm}[\cite{CortezPetite}, Proposition 7]\label{factor1}
If $\{H_n\}_{n=0}^{\infty}$ is a skeleton of a Toeplitz configuration $\und x\in\alf^G$, then $\hat G=\lim\limits_{\longleftarrow}(G/H_i,\pi_i)$ is the maximal equicontinuous factor of the minimal system $(\overline{G\und x}, G)$ and there is a factor map $\pi\colon \overline{G\und x}\to\hat G$ such that $\pi^{-1}(\{eH_i\}_{i=0}^{\infty})=\{\und x\}$.
\end{thm}

\begin{lem}[\cite{CortezPetite}, Lemma 2]\label{factor2}
Let $\hat G_1=\lim\limits_{\longleftarrow}(G/H^{(1)}_i,\pi_i)$ and $\hat G_2=\lim\limits_{\longleftarrow}(G/H^{(2)}_i,\pi_i)$ be two $G$-odometers. For $i\in\{1,2\}$ and $j\in\N$ let $\tilde e_i^{(j)}$ be the class of $e$ in $G/H^{(i)}_j$. Then the following conditions are equivalent:
\begin{enumerate}[leftmargin=*]
\item There is a factor map $\pi\colon (\hat G_1, G)\to (\hat G_2, G)$ such that $\pi(\{\tilde e^{(1)}_j\}_{j\in\N})=\{\tilde e^{(2)}_j\}_{j\in\N}$.
\item For every $i\in\N$ there exists $k\in\N$ such that $H^{(1)}_k\subset H^{(2)}_i$.
\end{enumerate}
\end{lem}

\begin{cor}\label{regularwell}
If $\und x\in\alf^G$ is a Toeplitz configuration and $\{H^{(1)}_i\}_{i\in\N}$, $\{H^{(2)}_i\}_{i\in\N}$ are skeletons of $\und x$, then for every $i\in\N$ there exists $k\in\N$ such that $H^{(1)}_k\subset H^{(2)}_i$.
\end{cor}

\begin{proof}
Let $\hat G_1$, $\hat G_2$ be $G$-odometers associated with $\{H^{(1)}_i\}_{i\in\N}$ and $\{H_i^{(2)}\}_{i\in\N}$, respectively.
It follows from Lemma~\ref{factor2} that to prove our claim it is enough to construct a factor map $\pi\colon (\hat G_1, G)\to(\hat G_2,G)$ such that $\pi(\{\tilde e^{(1)}_j\}_{j\in\N})=\{\tilde e^{(2)}_j\}_{j\in\N}$, where $\tilde e_j^{(i)}$'s are defined as in Lemma~\ref{factor2}.
For $i\in\{1,2\}$ let $\pi_i\colon(\overline{Gx},G)\to (\hat G_i,G)$ be a factor map provided by Theorem~\ref{factor1}. It follows from the Zorn lemma that there exists a compact set $F\subset \overline {Gx}$ such that $\pi_1(F)=\hat G_1$ but for every its proper compact subset $F'$ one has $\pi_1(F')\neq \hat G_1$. Define $\tilde \pi_1=\pi_1|_F$. Then $\tilde\pi_1$ is injective. Therefore $\tilde\pi_1$ is a homeomorphism and $\pi:=\pi_2\circ\tilde\pi_1^{-1}\colon \hat G_1\to\hat G_2$ satisfies requested conditions.
\end{proof}




For $g\in G$ and $m\in\N$ denote by $\tilde{g}^{(m)}$ an equivalence class of $g$ in $G/H_m$.
If  $\und x\in  \alf^G$ is a Toeplitz configuration such that \[\bigcup_{n\in\N}\text{Per}_{H_n}(\und x)=G,\]
then we say that $\und x$ is a \emph{Toeplitz configuration with respect to $\{H_n\}_{n\in\N}$}.
We denote by $\hat\mu$ the Haar measure on $\hat G$. 
Let $\phi\colon G\to\hat G$ be the natural homomorphism defined as
$\phi(g)=\left\{\tilde{g}^{(n)}\right\}_{n\in\N}$.

We denote by $[\tilde{a}^{(k)}]$ the $k$-cylinder of $\tilde{a}^{(k)}$, that is
\[[\tilde{a}^{(k)}]=\left\{\{a_n\}_{n\in\N}\in\hat{G}\,:\,a_k=\tilde{a}^{(k)}\right\}.\]
Let $\Pe_k$ be the clopen partition of $\hat G$ into $|G:H_k|$ $k$-cylinders.
For $f\colon \hat G\to \alf$ define $U(f)$ as the union of all sets in $\bigcup_{n\in\N}\Pe_n$ on which $f$ is constant.

The following proposition gives a characterization of Toeplitz configurations with respect to $\{H_n\}_{n\in\N}$.

\begin{prop}\label{haar}
 Let $f\colon \hat{G}\to \alf$ be a map satisfying $U(f)\supset \phi(G)$. Then the sequence given by $\eta(g)=f(\phi(g))$ is a Toeplitz configuration with respect to $\{H_n\}_{n\in\N}$.
Moreover, if $\{\eta(g)\}_{g\in G}$ is a Toeplitz configuration with respect to $\{H_n\}_{n\in\N}$, then there exists $f\colon \hat{G}\to \alf$ such that $U(f)\supset\phi(G)$ and $\eta(g)=f(\phi(g))$ for every $g\in G$.
\end{prop}

\begin{proof}
 Suppose that $U(f)\supset\phi(G)$. For $g\in G$, there exists $k\in\N$ and a cylinder $P\in \Pe_k$ containing $\phi(g)$ such that $f$ is constant on $P$.
 We claim that for every $h\in H_k$ one has $\eta(g)=\eta(hg)$. Clearly, $h\in H_k$ implies $\tilde h^{(i)}=\tilde{e}^{(i)}$ for every $i\leq k$.
Thus for every $P\in\mathcal P_k$ and $\und x=\{x_n\}_{n\in\N}\in P\subset\hat G$ we have $\phi(h)\cdot\und x=\{\tilde h^{(n)}x_n\}_{n\in\N}\in P$. 
  Then \[\eta(hg)=f(\phi(hg))=f(\phi(h)\phi(g))=f(\phi(g))=\eta(g).\]

Conversely, assume that $\eta\in\alf^G$ is a Toeplitz configuration with respect to $\{H_n\}_{n\in\N}$.
Let $g\in G$. Then $g\in\text{Per}_{H_k}(\eta, a)$ for some $k\in\N$ and $a\in\alf$. Define $f(x)=a$ for every $x\in[\tilde g^{(k)}]$. For all $x\in\hat G$ on which $f$ is not already defined let $f$ be arbitrary. To see that $f$ is well defined suppose that $x\in[\tilde g^{(k)}]\cap[\tilde\gamma^{(l)}]$ for some $g,\gamma\in G$ and $k\leq l$. Then $[\tilde\gamma^{(l)}]\subset [\tilde g^{(k)}]$ and so $\tilde\gamma^{(k)}=\tilde g^{(k)}$. Therefore $H_k\gamma=H_kg$ and $\eta(\gamma)=\eta(g)$. It is also clear that $\phi(G)\subset U(f)$. The proof is completed.
\end{proof}

\begin{rem}
The above lemma can be used to show that if $\und x$ is a Toeplitz configuration over $G$ and $G$ is a countable amenable residually finite group, then the set $\{g\in G\colon \und x(g)=1\}$ is a model set (cf.  \cite{BJL}).
\end{rem}


From now on we assume again that $G$ is amenable.
We call a Toeplitz configuration $\und x\in \alf^G$ \emph{regular}
if there exists a skeleton $\{H_n\}_{n\in\N}$ of $\und x$ such that
\[\sup_{n\in\N} D^*\big(\text{Per}_{H_n}(\und x)\big)=1.\]
\begin{rem}
The above definition is inspired by the notion of a regular Toeplitz configuration over $\mathbb Z$. It is known that such sequences are strictly ergodic and isomorphic in the measure theoretical category to their maximal equicontinuous factors, where in the Toeplitz configurations one considers the unique ergodic measure.
\end{rem}

\begin{rem}
By Corollary~\ref{regularwell} that if $\und x$ is regular, then for every skeleton $\{H_n\}_{n\in\N}$ one has 
\[\sup_{n\in\N} D^*\big(\text{Per}_{H_n}(\und x)\big)=1.\]
\end{rem}

\begin{rem}
If for every $N\in\N$ the group $G$ has only finitely many finite index subgroups of index at most $N$ then there exists a universal sequence of finite index subgroups $\{H_n\}_{n\in\N}$ such that for every $\Gamma<_fG$ there exists $n\in\N$ such that $H_n<\Gamma$. To see this enumerate finite index subgroups as $\Gamma_1,\Gamma_2,\ldots$ and define $H_n:=\Gamma_1\cap\ldots\cap\Gamma_n.$
\end{rem}

We say that $\und x\in \alf^G$ is \emph{periodic} if there exists $H<_fG$ such that $\text{Per}_{H}(\und x)=G.$

\begin{lem}\label{Z}
(cf. \cite[Lemma 2.6]{DHZ})
Let $F,K\in\Fin(G)$ and $\eps>0$. If $F$ is $(K,\eps/|K|)$-invariant, then 
$|\{s\in F\,:\,Ks\subset F\}|\geq (1-\eps)|F|.$
\end{lem}
A \emph{fundamental domain} for $H<_fG$ is a set $F$ such that for every $g\in G$ we have $|Hg\cap F|=1.$
\begin{lem}\label{sksksk}
Let $H<_fG$ and $K$ be a fundamental domain for $H$. If $B$ is a union of cosets of $H$, then $D^*(B)=D_*(B)=  |B\cap K| / |K|.$
\end{lem}

\begin{proof}
If $B$ is a union of cosets of $H$, then so is its complement. Hence it is enough to show that $D^*(B) = |B\cap K|/|K|$.   Let $\mathcal{F} = \{F_n\}_{n\in\N}$ be a F{\o}lner sequence. Fix $\eps >0$  and pick $N\in\N$ so that for every $n\geq N$, the set $F_n$ is $(K,\eps/|K|)-$invariant. Define $S=\{s\in F_n\,:\,Ks\subseteq F_n\}$. By Lemma~\ref{Z} we get
\begin{equation}\label{eq1} |S|\geq (1-\eps)|F_n|.\end{equation}
Since $K$ is a fundamental domain one has\begin{equation}\label{eq2} \sum_{s\in S} |(Ks)\cap B|=|S||K\cap B|\end{equation}
On the other hand $g\in F_n$ belongs to at most $|K|$ translates of $K$.
Thus \begin{equation}\label{eq3} \sum_{s\in S} |(Ks)\cap B| \leq |K||B\cap F_n| .\end{equation} 
Combining (\ref{eq1}), (\ref{eq2}) and (\ref{eq3}) we obtain 
\[ (1-\eps)|F_n||K\cap B| \leq |K||F_n\cap B| . \]
Dividing both sides by $|F_n||K|$ we get \[ (1-\eps) |K\cap B|/|K|\leq |F_n\cap B|/|F_n|. \]
Taking upper limit on the right hand side leads to \[ (1-\eps) |K\cap B|/|K|\leq \bar d_{\mathcal F} (B) .\]
Since $\eps>0$ and $\mathcal{F}$ are arbitrary, we see
\[ |K\cap B|/|K|\leq \bar d_{\mathcal F} (B)\leq \sup_{\mathcal F} \bar d_{\mathcal F} (B)  = D^*(B) .\]
On the other hand for every $t\in G$ we have 
\[ D_K(B) = D_{Kt}(B) = |K\cap B|/|K|. \] 
\begin{flalign*}
\text{Hence }&\quad D^*_K(B) = |K\cap B|/|K| &
\end{flalign*}
\begin{flalign*}
\text{and }&\quad
D^*(B) = \inf_{F\in\text{Fin}(G)} D^*_F(B)\leq D^*_K(B) = |K\cap B|/|K|.& \qedhere
\end{flalign*}
\end{proof}

\begin{cor}\label{**}
For every $j\in\N$ and $\und x\in \alf^G$ the set $\text{Per}_{H_j}(\und x)$ has Banach density.
This is because $\text{Per}_{H_j}(\und x)$ consists of $m$ cosets of $H_j$ for some $0\leq m\leq |G:H_j|$.
\end{cor}


\begin{lem}\label{sss}
Every regular Toeplitz configuration is a quasi uniform limit of a sequence of periodic sequences.
\end{lem}

\begin{proof}
Fix a regular Toeplitz configuration $\und x=\{x_g\}_{g\in G}$. Let $\{H_n\}_{n\in\N}$ be a skeleton of $\und x$.
Note that $\text{Per}_{H_m}(\und x)\subset\text{Per}_{H_n}(\und x)$ for $m\leq n$.
Therefore \[1=\sup_{m\geq 1}D^*(\text{Per}_{H_m}(\und x))=\lim\limits_{m\to\infty}D^*(\text{Per}_{H_m}(\und x)).\]
Let $F_n$ be a fundamental domain for $H_n$ in $G$. Define $\und x^{(n)}\in \alf^G$ as $\und x^{(n)}(H_nf)=\und x(f)$ for every $f\in F_n$.  
Then \[\{g\in G\,:\,\und x^{(n)}(g)\neq\und x(g)\}\subset G\setminus \text{Per}_{H_n}(\und x).\]
Thus $D^*(\und x^{(n)},\und x)\leq 1- D^*(\text{Per}_{H_n}(\und x))\to 0$ as $n\to\infty$.
\end{proof}

Since every periodic sequence $\und x\in \alf^G$  satisfies $h(\und x)=0$ and $|\M_G(\und x)|=1$, we get the following corollary of Lemma~\ref{entropiaciagla} and Theorem~\ref{thm:ciagwe'}:


\begin{cor}
If  $\und x\in \alf^G$ is a regular Toeplitz configuration, then the system $(\overline{G\und x},G)$ is uniquely ergodic and has zero topological entropy.
\end{cor}

We need the following lemma which is a reformulation of \cite[Lemma 4]{Cortez} (note that the condition~(\ref{problem}) is slightly different in our case as the authors of~\cite{Cortez} consider right 
F{\o}lner sequences and we use left ones).
\begin{lem}\label{Cortez}
If $G$ is an amenable residually finite group then there exist
a sequence $\{H_n\}_{n=0}^{\infty}$ with $H_n\lhd_f G$ and  a F{\o}lner sequence $\{F_n\}_{n\in\N}$ satisfying:
\begin{enumerate}[leftmargin=2.5em]
\item $\displaystyle G=H_0\supset H_1\supset H_2\supset\ldots\text{ and }\bigcap_{n=0}^{\infty}H_n=\{e\},$
\item $\displaystyle\{e\}=F_0\subset F_1\subset F_2\subset\ldots\text{ and }\bigcup_{n=0}^{\infty}F_n=G,$
\item for each $n\in\N$ the set $F_n$ is a fundamental domain for $G/H_{n}$,
\item\label{problem} $\displaystyle F_{i+1}=\bigsqcup_{v\in F_{i+1}\cap H_{i}}F_iv\text{ for every }i\in\N$,
\end{enumerate}
\end{lem}
%
To the end of the paper we fix a F{\o}lner sequence $\{F_n\}_{n\in\N}$ which satisfies the conditions of Lemma~\ref{Cortez}.



\begin{rem}
If $\und x\in\alf^G$ is Toeplitz and $f\colon\hat G\to\alf$ is as in Proposition~\ref{haar}, then $\und x$ is regular if and only if $\hat\mu(U(f))=1$. To see this note that \begin{equation*}\hat\mu(U(f))=\lim\limits_{k\to\infty}|\{P\in\Pe_k\,:\,P\subset U(f)\}|/|F_k|=\lim\limits_{k\to\infty}D^*(\text{Per}_{H_k}(\und x)),\end{equation*}
where the last equality follows from Lemma~\ref{sksksk}.
\end{rem}

\subsection{Path Connectedness}
Recall that a set $A$ is \emph{path-connected} if for all $x,y\in A$ there exists a continuous function $f\colon [0,1]\to A$ such that $f(0)=x$ and $f(1)=y$.
We prove that the family of all Toeplitz configurations is path-connected. 

\begin{lem}\label{lematD}
If $G$ is an amenable residually finite group, then there exists a function $\Psi\colon[0,1]\to\{0,1\}^G$ such that $\Psi(0)=0^G$, $\Psi(1)=1^G$ and for every $s,t\in[0,1],$ $\Psi(s),\Psi(t)$ are Toeplitz configurations satisfying $D^*(\Psi(s),\Psi(t))\leq|s-t|$.
\end{lem}

\begin{proof}
Fix $t\in[0,1]$. Let $k_0=|F_1|$. We define $\Psi(t)=\und x^{(t)}=\{x^{(t)}_g\}_{g\in G}$ inductively. Let 
$q_0:=\max\{0\leq l\leq k_0\,:\,l/k_0\leq t\}.$
Enumerate points in $F_1$ as $f_1^{(1)},\ldots, f^{(1)}_{k_0}$. 
Define\[D_1(t):=\bigcup_{j=1}^{q_0}f_j^{(1)}H_1\quad \text{ and } \quad E_1(t):=\bigcup_{j=r}^{k_0}f_j^{(1)}H_1, \text{ where } r= \left\{ \begin{array}{l} q_0+1 \text{ if } t=q_0/k_0, \\ q_0+2, \text{ otherwise.}\end{array}\right .\]

Note that:
\begin{inparaenum}
\item\label{1fff}  it is possible that $D_1(t)=\emptyset$ or $E_1(t)=\emptyset$,
\item\label{222} $D_1(t)\cup E_1(t)=G$ if and only if $t=q_0/k_0$,
\item $D_1(t)\cap E_1(t)=\emptyset$,
\item $|F_1\setminus (D_1(t)\cup E_1(t))|\leq 1$,
\item\label{5fff} $D_1(t)$ and $E_1(t)$ are union of cosets of $H_1$.
\end{inparaenum}
By (\ref{222}) above our construction is finished if $t=q_0/k_0$, otherwise we carry on.
Assume we have defined disjoint sets $D_n(t)$ and $E_n(t)$ for some $n\geq 1$ and $F_n\setminus (D_n(t)\cup E_n(t))=\{f^{(n)}_*\}$.
We enumerate points in the set \[F_{n+1}\setminus(D_n(t)\cup E_n(t))=\bigsqcup_{v\in H_n\cap F_{n+1}}\{vf^{(n)}_*\}\] as $f_1^{(n+1)},\ldots,f_{k_n}^{(n+1)}$.
Note that it follows from Lemma~\ref{sksksk} that $t-D^*(D_n(t))\geq 0$. Let $q_n:=\max\{0\leq l\leq k_n\,:\,l/k_n\leq t-D^*(D_n(t))\}$.
Define 
\begin{multline*}
D_{n+1}(t)=D_n(t)\cup\bigcup_{j=1}^{q_n}f_j^{(n+1)}H_{n+1},\quad
E_{n+1}(t)=E_n(t)\cup\bigcup_{j=r}^{k_n}f_j^{(n+1)}H_{n+1}, \\ \text{ where }
 r=\left\{ \begin{array}{l} q_n+1 \text{ if }D^*(D_n(t))+q_n/k_n=t \\ q_n+2, \text{ otherwise.}\end{array}\right.
.\end{multline*}
 Note that statements analogous to (\ref{1fff})--(\ref{5fff}) above hold for $D_{n+1}(t)$ and $E_{n+1}(t)$.

At the end we have defined (possibly finite) sequences $D_1(t)\subset D_2(t)\subset\ldots$ and $E_1(t)\subset E_2(t)\subset\ldots$ such that $G=D(t)\cup E(t)$, where $D(t)=D_1(t)\cup D_2(t)\cup\ldots$ and $E(t)=E_1(t)\cup E_2(t)\cup\ldots$. Clearly $D(t)$ and $E(t)$ are unions of cosets and are disjoint. Therefore $\und x^{(t)}=\{x_g^{(t)}\}_{g\in G}$ given by 
 \begin{equation*}
x^{(t)}_g=\left\{\begin{array}{lll} 1 &\text{ if }&g\in D(t),\\
0&\text{ if }&g\in E(t).
\end{array}\right.
\end{equation*}
is a Toeplitz configuration. Hence the function $\Psi(t)=\und x^{(t)}$ satisfies the first two claims.
To justify the third claim fix $0\leq s<t\leq 1$ and for each $n\geq1$ construct $D_n(s)\subseteq D_n(t)$ as described above. 
Then $D^*\left(\und x^{(s)},\und x^{(t)}\right)=D^*(\Gamma)$, where $\Gamma:=\{g\in G\,:\,x^{(s)}_g\neq x^{(t)}_g\}$, and 
\[
D^*(\Gamma)=\lim\limits_{n\to\infty}\sup_{g\in G} D_{F_ng}(\Gamma)\leq\lim\limits_{n\to\infty} D^*\big(D_n(t)\setminus D_n(s)\big)+1/|F_n|=t-s.\qedhere
\]
\end{proof}
Theorem~\ref{connected2} extends \cite[Proposition 6]{DI} where the analogous statement is proved for Toeplitz configurations indexed by the group of integers.
\begin{thm}\label{connected2}
The family of Toeplitz configurations is $D_W$-path-connected.
\end{thm}

\begin{proof}
Let $\alf\subset\mathbb R$ be a finite set.
Pick two Toeplitz configurations $\und z=\{z_g\}_{g\in G}$ and $\und z'=\{z'_g\}_{g\in G}\in \alf^G$. 
We will construct a path $\{u^{(t)}\,:\,t\in[0,1]\}\subset \alf^G$ connecting $\und z$ with $\und z'$.
Let $\Psi\colon[0,1]\to\{0,1\}^G$ and $\und x^{(t)}$ for $t\in[0,1]$ be defined as in Lemma~\ref{lematD}. 
For $t\in[0,1]$ define $\und u^{(t)}=\{u^{(t)}_g\}_{g\in G}$ as ${u}^{(t)}_g=
x^{(t)}_g\und{z}_g+(1-x^{(t)}_g) \und{z}'_g.$
Clearly the alphabet over which $\und u^{(t)}$ is defined does not depend on $t$,  $D^*(\und u^{(s)},\und u^{(t)})\leq D^*(\und x^{(s)},\und x^{(t)})$ and $D^*(\und x^{(s)},\und x^{(t)})\to 0\text{ as }s\to t.$ Fix $t\in[0,1]$. We will show that $\und u^{(t)}$ is a Toeplitz configuration. To this end fix $g\in G$. Let $H<_fG$ be such that $|G:H|<\infty$ and $g\in\text{Per}_{H}(\und x^{(t)})\cap \text{Per}_{H}(\und z)\cap \text{Per}_{H}(\und z')$ ($H$ is the intersection of periods of $g$ for $\und x^{(t)}$, $\und z$ and $\und z'$). Then for every $h\in H$ one has \[{u}^{(t)}_{gh}=
x^{(t)}_{gh} \und{z}_{gh} +(1-x^{(t)}_{gh}) \und{z}'_{gh}=x^{(t)}_g \und{z}_g+(1-x^{(t)}_g) \und{z}'_g={u}^{(t)}_g\]
and so $g\in\text{Per}_H(\und u^{(t)})$. This finishes the proof.
\end{proof}

From the above proof one can also see that:

\begin{thm}\label{connected}
The space of all Toeplitz configurations with respect to a fixed nested sequence $\{H_n\}_{n\in\N}$ is path-connected with respect to $D_W$.
\end{thm}


\subsection{The Proof of the Krieger Theorem}
The above considerations lead to the alternative proof of the Krieger theorem \cite[Theorem 1.1.]{KriegerFR}.

\begin{thm}
Let $G$ be a countable amenable residually finite group and $\alf$ be a finite set. Then  for every number $h\in[0,\log |\alf|)$ there exists a Toeplitz configuration $\eta\in \alf^G$ such that $h(\eta)=h$.
\end{thm}

\begin{proof}
Fix $\gamma\in(0,1)$.
By Theorem~\ref{connected} and Theorem~\ref{entropiaciagla} it is enough to show that there is a Toeplitz configuration $\eta\in\alf^G$ with $h(\eta)\geq\gamma\log|\alf|$. 
Let $\{F_n\}_{n\in\N}$ be a \Folner sequence given by Lemma~\ref{Cortez}.
We define $\{k_n\}_{n\in\N}\subset\N$ inductively: $k_0=0$, given $k_n$ for some $n\in\N$ we pick $k_{n+1}$ such that $F_{k_{n+1}}$ contains at least $|\alf|^{|F_{k_n}|}$ copies of $F_{k_n}$. For any $n\in\N$ let $r_n=\lfloor (1-\gamma)|F_{k_n}|/2^n\rfloor$. Choose a sequence $\{G_n\}_{n\in\N}$ of subsets of $G$ such that:
\begin{inparaenum}
\item $G_n\subset F_{k_n}$ and $|G_n|=r_n$ for every $n\in\N$,
\item  $ G_i\cap\bigcup\{\tilde{g}^{(k_j)}\colon g\in G_j\} =\emptyset$ for $i\neq j,$
\item $\bigcup\{ \bigcup\{\tilde{g}^{(k_n)}\colon g\in G_n\}\colon n\in\N\}=G.$
\end{inparaenum}
 Now construct $\eta\in\alf^G$ as follows. 
 On $F_0$ let $\eta$ be arbitrary. For every $g\in G_0$ and $h\in \tilde{g}^{(k_0)}$ let $\eta(h)=\eta(g)$.
 Assume that for some $n\in\N$ we have defined $\eta$ on $ T_n$, where 
 \[T_n:=\bigcup_{j=0}^n\bigcup_{g\in G_j}\tilde{g}^{(k_j)}.\] 
This means that in $F_{k_n}$ there are $\sum_{i=0}^nr_i\cdot |H_{k_{n}}:H_{k_i}|$ elements over which $\eta$ is already constructed. 
Pick \[|\alf|^{|F_{k_n}|-\sum_{i=0}^nr_i\cdot |H_{k_{n}}:H_{k_i}|}\] copies of $S_n:=F_{k_n}\setminus T_n$ in $F_{k_{n+1}}$. 
On each of these copies define $\eta$ in a different way such that every configuration over $S_n$ can be found in $F_{k_{n+1}}$. 
This in particular means that 
\[|\mathcal B_{F_{k_n}}|\geq |\alf|^{|F_{k_n}|-\sum_{i=0}^nr_i\cdot |H_{k_{n}}:H_{k_i}|}\geq|\alf|^{\gamma|F_{k_n}|},\]
where the last inequality follows from the definition of $r_i$: \[\sum_{i=0}^nr_i\cdot |H_{k_{n}}:H_{k_i}|\leq \sum_{i=0}^n(1-\gamma)(|F_{k_i}|/2^i)(|F_{k_n}|/|F_{k_i}|)\leq(1-\gamma)|F_{k_n}|.\]
For every $g\in G_{n+1}$ and $h\in \tilde{g}^{(k_{n+1})}$ put $\eta(h)=\eta(g)$.

It is easy to see that $\eta$ is a Toeplitz configuration. Moreover, one has 
\[h(\eta)\geq\limsup\limits_{n\to\infty}\frac{1}{|F_{k_n}|}\log|\mathcal B_{F_{k_n}}|\geq \gamma\log|\alf|.\qedhere\]
\end{proof}
\section *{Acknowledgments}
We would like to thank our advisor Dominik Kwietniak for many helpful comments, constant support and encouragement. We are also grateful to Jakub Byszewski and Samuel Petite for their remarks and suggestions. 

The research of Marta Straszak was supported by the National
Science Centre (NCN) under grant no.
DEC-2012/07/E/ST1/00185.
Martha \L{}\k{a}cka acknowledges support of the National Science Centre (NCN), grant no. 2015/19/N/ST1/00872.

\def\cprime{$'$} \def\ocirc#1{\ifmmode\setbox0=\hbox{$#1$}\dimen0=\ht0
  \advance\dimen0 by1pt\rlap{\hbox to\wd0{\hss\raise\dimen0
  \hbox{\hskip.2em$\scriptscriptstyle\circ$}\hss}}#1\else {\accent"17 #1}\fi}
  \def\ocirc#1{\ifmmode\setbox0=\hbox{$#1$}\dimen0=\ht0 \advance\dimen0
  by1pt\rlap{\hbox to\wd0{\hss\raise\dimen0
  \hbox{\hskip.2em$\scriptscriptstyle\circ$}\hss}}#1\else {\accent"17 #1}\fi}


\begin{thebibliography}{10}


\bibitem{Auslander59}
J.~Auslander.
\newblock Mean-{$L$}-stable systems.
\newblock {\em Illinois J. Math.}, 3:566--579, 1959.





\bibitem{BKKPL}
A.~{Bartnicka}, S.~{Kasjan}, J.~{Ku{\l}aga-Przymus}, and M.~{Lema{\'n}czyk}.
\newblock {$\mathscr{B}$-free sets and dynamics}.
\newblock To appear in: TAMS. The preprint available at:
  https://arxiv.org/abs/1509.08010.

\bibitem{BBF}
M. Beiglb{\"o}ck, V. Bergelson, and A. Fish.
\newblock Sumset phenomenon in countable amenable groups.
\newblock {\em Adv. Math.}, 223(2):416--432, 2010.


\bibitem{BJL}
M. Baake, T. J{\"a}ger, and D. Lenz.
\newblock Toeplitz flows and model sets.
\newblock {\em Bull. Lond. Math. Soc.}, 48(4):691--698, 2016.

\bibitem{BFK}
F.~Blanchard, E.~Formenti, and P.~Kůrka.
\newblock Cellular automata in the {C}antor, {B}esicovitch, and {W}eyl
  topological spaces.
\newblock {\em Complex Systems}, 11(2):107--123, 1997.



\bibitem{CortezPetite}
M. I. Cortez and S. Petite.
\newblock {$G$}-odometers and their almost one-to-one extensions.
\newblock {\em J. Lond. Math. Soc. (2)}, 78(1):1--20, 2008.

\bibitem{Cortez}
M. I. Cortez and S. Petite.
\newblock Invariant measures and orbit equivalence for generalized {T}oeplitz
  subshifts.
\newblock {\em Groups Geom. Dyn.}, 8(4):1007--1045, 2014.



\bibitem{DFR}
T.~{Downarowicz}, B.~{Frej}, and P.-P. {Romagnoli}.
\newblock {Shearer's inequality and Infimum Rule for Shannon entropy and
  topological entropy}.
\newblock In {\em Dynamics and Numbers, Contemporary Mathematics, vol. 669},
  pages 76--89. AMS, 2016.


\bibitem{DHZ}
T.~{Downarowicz}, D.~{Huczek}, and G.~{Zhang}.
\newblock {Tilings of amenable groups}.
\newblock To appear in \emph{Journal für die reine und angewandte Mathematik}, DOI: 10.1515/crelle-2016-0025.

\bibitem{DI}
T.~Downarowicz and A.~Iwanik.
\newblock Quasi-uniform convergence in compact dynamical systems.
\newblock {\em Studia Math.}, 89(1):11--25, 1988.


\bibitem{Downar1}
T. Downarowicz.
\newblock The {C}hoquet simplex of invariant measures for minimal flows.
\newblock {\em Israel J. Math.}, 74(2-3):241--256, 1991.


\bibitem{Fomin}
S.~Fomin.
\newblock On dynamical systems with a purely point spectrum.
\newblock {\em Doklady Akad. Nauk SSSR (N.S.)}, 77:29--32, 1951.

\bibitem{FGJ}
G.~Fuhrmann, M.~Gr{\"o}ger, and T.~J{\"a}ger.
\newblock Amorphic complexity.
\newblock {\em Nonlinearity}, vol. 29, No. 2, 2016.

\bibitem{GJ}
R. Gjerde and {\O}. Johansen.
\newblock Bratteli-{V}ershik models for {C}antor minimal systems: applications
  to {T}oeplitz flows.
\newblock {\em Ergodic Theory Dynam. Systems}, 20(6):1687--1710, 2000.


\bibitem{JK}
K.~Jacobs and M.~Keane.
\newblock {$0-1$}-sequences of {T}oeplitz type.
\newblock {\em Z. Wahrscheinlichkeitstheorie und Verw. Gebiete}, 13:123--131,
  1969.

\bibitem{KriegerFR}
F. Krieger.
\newblock Sous-d\'ecalages de {T}oeplitz sur les groupes moyennables
  r\'esiduellement finis.
\newblock {\em J. Lond. Math. Soc. (2)}, 75(2):447--462, 2007.

\bibitem{Krieger}
F. Krieger.
\newblock Toeplitz subshifts and odometers for residually finite groups.
\newblock In {\em \'{E}cole de {T}h\'eorie {E}rgodique}, volume~20 of {\em
  S\'emin. Congr.}, pp. 147--161. Soc. Math. France, Paris, 2010.

\bibitem{KLOg}
D.~{Kwietniak}, M.~{{\L}{\c a}cka}, and P.~{Oprocha}.
\newblock {Generic Points for Dynamical Systems with Average Shadowing}.
\newblock {\em Monatsh. Math.}, 183, 2017.

\bibitem{Lindenstrauss01}
E. Lindenstrauss.
\newblock Pointwise theorems for amenable groups.
\newblock {\em Invent. Math.}, 146(2):259--295, 2001.

\bibitem{LW}
E. Lindenstrauss and B. Weiss.
\newblock Mean topological dimension.
\newblock {\em Israel J. Math.}, 115:1--24, 2000.



\bibitem{Oxtoby52}
J. C. Oxtoby.
\newblock Ergodic sets.
\newblock {\em Bull. Amer. Math. Soc.}, 58:116--136, 1952.


\bibitem{ST}
V. Salo and I. T{\"o}rm{\"a}.
\newblock Geometry and dynamics of the {B}esicovitch and {W}eyl spaces.
\newblock In {\em Developments in language theory}, {\em Lecture
  Notes in Comput. Sci. vol. 7410}, pp. 465--470. Springer, 2012.

\bibitem{Schaefer}
P. Schaefer.
\newblock Limit points of bounded sequences.
\newblock {\em Amer. Math. Monthly}, 75:51, 1968.

\bibitem{Shields}
P. C. Shields.
\newblock {\em The ergodic theory of discrete sample paths},  {\em
  Graduate Studies in Mathematics}, vol. 13,
\newblock American Mathematical Society, Providence, RI, 1996.



\bibitem{Williams}
S. Williams.
\newblock Toeplitz minimal flows which are not uniquely ergodic.
\newblock {\em Z. Wahrsch. Verw. Gebiete}, 67(1):95--107, 1984.


\end{thebibliography}
\end{document}